\documentclass[12pt]{article}
\usepackage{amsmath, graphicx, amsfonts,amssymb, calrsfs}
\usepackage{amsfonts,mathrsfs, color, amsthm}
\usepackage{bm}
\usepackage{dsfont}
\usepackage{subcaption}
\usepackage{bbm}
\usepackage{cases}
\usepackage{float}

\usepackage{enumitem}
\usepackage{hyperref}
\hypersetup{
colorlinks   = true, 
urlcolor     = blue, 
linkcolor    = cyan, 
citecolor   = red 
}

\def\sphere{S^{n-1}}

\def\Rn{{\mathbb R^n}}
\def\R{\mathbb{R}}

\def\A{\mathds{A}}
\def\B{\mathds{B}}
\def\P{\mathbb{P}}
\def\Oc{\Omega_{C^\circ}}
\def\copolar{\A^{\diamond}_C}

\newtheorem{theorem}{Theorem}[section]
\newtheorem{prop}[theorem]{Proposition}

\newtheorem{problem}[theorem]{Problem}

\newtheorem{lemma}[theorem]{Lemma}
\newtheorem{remark}[theorem]{Remark}
\newtheorem{proposition}[theorem]{Proposition}
\newtheorem{corollary}[theorem]{Corollary}

\newtheorem{definition}[theorem]{Definition}
\numberwithin{equation}{section}

\def\sphere{S^{n-1}}
\def\vol{V_n}

\def\bt{\begin{theorem}}
\def\et{\end{theorem}}
\def\be{\begin{equation}}
\def\ee{\end{equation}}
\def\bl{\begin{lemma}}
\def\el{\end{lemma}}
\def\br{\begin{remark}}
\def\er{\end{remark}}
\def\bc{\begin{corollary}}
\def\ec{\end{corollary}}
\def\bd{\begin{definition}}
\def\ed{\end{definition}}
\def\bp{\begin{proposition}}
\def\ep{\end{proposition}}

\def\Ln{\mathscr{L}_C}

\allowdisplaybreaks[4]

\begin{document}
\title{The $L_p$ dual Minkowski problem for unbounded closed convex sets
}
\author{Wen Ai, Yunlong Yang and Deping Ye}
\date{}
\maketitle
	
\begin{abstract}
The central focus of this paper is the $L_p$ dual Minkowski problem for $C$-compatible sets, where $C$ is a pointed closed convex cone in $\mathbb{R}^n$ with nonempty interior. Such a problem deals with the characterization of the $(p, q)$-th dual curvature measure of a $C$-compatible set. It produces new
Monge-Amp\`{e}re equations for unbounded convex hypersurface, often defined over open domains and with non-positive unknown convex functions. 
Within the family of $C$-determined sets, the $L_p$ dual Minkowski problem is solved for $0\neq p\in \mathbb{R}$ and $q\in \mathbb{R}$; while it is solved for the range of 
$p\leq 0$ and $p<q$
within the newly defined family of $(C, p, q)$-close sets. When $p\leq q$, we also obtain some results regarding the uniqueness of solutions to the $L_p$ dual Minkowski problem for $C$-compatible sets. 
\end{abstract}

\section{Introduction}
The Brunn-Minkowski theory of convex bodies (i.e., compact convex subsets of $\Rn$ with nonempty interiors) is a fundamental area in mathematics originating from the combination of volume and Minkowski addition, which leads to many far-reaching results, such as the Brunn-Minkowski inequality, the Minkowski first
inequality, the Alexandrov-Fenchel inequality, the classical Minkowski problem and its solutions, etc
(see Schneider \cite{Sch14} for a comprehensive introduction of this theory). These results have found important applications in various areas of mathematics, e.g.,  analysis, geometry, probability, and many
more. Dual to the Brunn-Minkowski theory of convex bodies,  Lutwak in \cite{L75, L88} defined the radial additions of convex bodies and the intersection bodies, and proved the dual Brunn-Minkowski and Minkowski inequalities for convex bodies. (Note that the dual theory also works for star bodies). These results lead to the resolution of the famous Busemann-Petty problem \cite{Gar94, GKS99, Zhang99}. Recently, a breakthrough in the dual Brunn-Minkowski theory has been made by Huang, Lutwak, Yang and Zhang in \cite{HLYZ2016},  where the dual curvature measures of convex bodies were obtained by the variation of the $q$-th dual volume 
in terms of
the Minkowski addition. The introduction of the dual curvature measures of convex bodies not only gives the long-time missing geometric measures in the dual  Brunn-Minkowski theory, but also leads to the dual Minkowski problem for convex bodies, which at the first time brings the radial functions of convex bodies into the Minkowski type problems. Since then, many important contributions to the dual Minkowski problem have been made, see e.g., \cite{BHP18, BLYZZ19, CL18, HP18, HJ19, JW17, LSW20, WFZ19, WZ20, Zhao17, Zhao18}.

The above-mentioned Brunn-Minkowski theory and its dual theory deal with compact convex sets in $\Rn$. A natural question to ask is:  
\textit{what are the analogous theories for unbounded convex sets in $\Rn$?}
Unbounded closed convex sets have demonstrated their fundamental significance in various areas, including differential geometry, partial differential equations, commutative algebra, and singularity theory, see e.g., \cite{KK14,R00,R01,VK85}.  The first step towards the Brunn-Minkowski theory for unbounded convex sets in $\Rn$ was made by Schneider in \cite{Sch18} 
following from the pioneer works \cite{KT} by Khovanski\u{\i} and Timorin, and \cite{MR} by Milman and Rotem; while the dual Brunn-Minkowski theory for unbounded convex sets in $\Rn$ was recently started by Li, Ye and Zhu in \cite{LYZ22}. The unbounded convex subsets of interest are the $C$-compatible sets (including the $C$-close sets and $C$-full sets by Schneider \cite{Sch18,Sch21}) introduced in \cite{LYZ22}. Hereafter, we call  $C\subset \mathbb{R}^n$ a  pointed closed convex cone, if
$C=\{ \lambda x : \lambda \geq 0\,\, \text{and}\,\, x \in C\}$  is a closed and convex set such that $C$ has nonempty interior and $C \cap (-C)=\{o\}$ with $o$ being the origin of $\mathbb{R}^n$. The polar cone $C^{\circ}$  of a pointed closed convex cone $C$ is given by
\begin{align} 
C^\circ=\left\{ x \in \mathbb{R}^n : x \cdot y \leq 0 \,\,\, \text{for all}\,\,\, y \in C \right\}, \label{def-polar-cone} \end{align}  where $x \cdot y$ represents the inner product
of $x,y\in\mathbb{R}^n$. An unbounded closed convex set $o\notin E\subset C$ is called $C$-compatible \cite{LYZ22} if \begin{eqnarray}
E=C\cap  \bigcap_{u\in \Oc}  \Big\{H_u^{-}:   E\subseteq H_u^-\Big\}. \label{def-1-1-2}
\end{eqnarray} Hereafter, $\Oc=\sphere \cap \text{int}C^\circ$ is the intersection of $S^{n-1}$, the unit sphere in $\Rn$, and 
$\text{int} C^{\circ}$, the interior of $C^{\circ}$. We also use $H^-_u$ for  lower half space with unit outer normal $u\in \sphere$. 
By a $C$-close set, we mean a $C$-compatible set $\A$ with  $\vol(C \,\backslash \A)$ being finite. If $\A$ is a $C$-close set, then 
its complement set $A=C\,\backslash \A$ is called a $C$-coconvex set. If in addition $A=C\,\backslash \A$ is bounded, then $\A$ will be called a $C$-full set. 

Let $\Ln$ be  the collection of $C$-compatible sets. It will play essential roles in the dual theory of unbounded convex sets in $\mathbb{R}^n$. For $\A\in \Ln$, its support and radial functions, denoted by $h_C(\A, \cdot):\Oc\rightarrow (- \infty, 0)$ and $\rho_C(\A, \cdot): \Omega_C\rightarrow (0, \infty]$, can be defined by  
\begin{align}
h_C(\A, u)&=\sup\{x \cdot u: x \in \A\}\,\,\,\text{for}\,\,\,u\in \Oc, \label{support-1-A} \\ 
\rho_C(\A, v)&=\sup\{r>0: rv\in C\,\backslash \A\}\,\,\,\text{for}\,\,\,v\in \Omega_C. \label{radial-1-A}\end{align}
Here $\Omega_C=\sphere \cap \text{int}C$. 
With the help of radial function,  Li, Ye and Zhu \cite{LYZ22}  defined the $q$-th dual volume of $C$-compatible set $\A\in \Ln$ as follows:
\begin{equation*} 
\widetilde{V}_q(\A)=\frac{1}{n}\int_{\Omega_C}\rho_C(\A, v)^qdv
\end{equation*}  for $0\neq q \in \mathbb{R}$.  If $\widetilde{V}_q(\A)<\infty$, then $\A$ is called a $(C, q)$-close set. In particular, when  $q=n$, $\widetilde{V}_n(\A)=V_n(C\,\backslash\A)$ is the volume of $C\,\backslash\A$, and hence a $(C, n)$-close set is a $C$-close set. 

The $L_p$ addition of $C$-coconvex sets, called the $p$-co-sum \cite{Sch18, YYZ22}, is used as an algebraic operation to form new $C$-coconvex sets. 
Let $\A_1$ and $\A_2$ be two $C$-close sets. 
For $p\in (0, 1]$,
the $p$-co-sum of $C$-coconvex sets $A_1=C\,\backslash \A_1$ and $A_2=C\,\backslash \A_2$, denoted by $A_1 \oplus_p A_2$, is given as the following:
\begin{equation}
C\,\backslash (A_1 \oplus_p A_2)=C \cap \bigcap_{u\in \Oc} \big\{ x \in \mathbb{R}^n: x\cdot u\leq -\overline{h}_C(A_1 \oplus_p A_2, u)\big\}, \label{co-p-sum-def-1}
\end{equation}
where $\overline{h}_C(A, u)=-h_C(\A, u)\geq 0$
denotes the support function of $A=C\,\backslash \A$, and 
\begin{align} \overline{h}_C(A_1 \oplus_p A_2, u)=[\overline{h}_C(A_1, u)^p+\overline{h}_C(A_2, u)^p]^\frac{1}{p}\,\,\,\text{for}\,\,\, u \in \Oc \label{co-p-sum-def-2} \end{align}
is the support function of $A_1 \oplus_p A_2$.  
When $p=1$, it is the co-sum of $C$-coconvex sets formulated by Schneider \cite{Sch18}, written as $A_1 \oplus A_2$. Fundamental results in the $L_p$ Brunn-Minkowski theory for $C$-coconvex sets include, for example,  the $L_p$ Brunn-Minkowski inequality for $p=1$ by Schneider \cite{Sch18} and for $p\in (0, 1)$ by Yang, Ye, and Zhu \cite{YYZ22}:
\begin{equation}\label{LpBMC}
V_n(A_1 \oplus_p A_2)^\frac{p}{n} \leq V_n(A_1)^\frac{p}{n}+V_n(A_2)^\frac{p}{n},
\end{equation}
with equality if and only if $A_1=\alpha A_2$ for some $\alpha>0$. A consequence of  \eqref{LpBMC} is the $L_p$ Minkowski inequality for $C$-coconvex sets, for $p=1$ by Schneider \cite{Sch18} and $p\in (0, 1)$ by Yang, Ye, and Zhu \cite{YYZ22}: 
\begin{equation}\label{LpMC}
\overline{V}_p(A_1, A_2)=\frac{1}{n}\int_{\Oc} \overline{h}_C(A_2, u)^p d\overline{S}_{n-1, p} (A_1, u)
\leq V_n(A_1)^\frac{n-p}{n} V_n(A_2)^\frac{p}{n},
\end{equation}
with equality if and only if $A_1=\alpha A_2$ for some $\alpha>0$. Here,  $\overline{S}_{n-1, p} (A, \cdot)=S_{n-1, p} (\A, 
\cdot)$ is the $L_p$ surface area measures of $A$ and $\A$, which can be formulated by (say for $\A$): for $p=1$ \cite{Sch18},  
\begin{equation*}
S_{n-1}(\A, \eta)=\mathscr{H}^{n-1}(\nu^{-1}_\A(\eta))\ \ \ \mathrm{for\ every\ Borel\ set\ } \eta \subseteq \Oc,
\end{equation*} where $\mathscr{H}^{n-1}$ denotes the $(n-1)$-dimensional Hausdorff measure and $\nu^{-1}_\A$ is the reverse Gauss map of $\A$, and  for $p\in \R$ ($p=0$ in \cite{Sch18} and $p\neq 0$ in \cite{YYZ22}), 
\begin{equation}\label{Lp-measure-support}
dS_{n-1, p}(\A, \cdot)=(-h_C(\A, \cdot))^{1-p}dS_{n-1}(\A, \cdot). 
\end{equation} It is worth to mention that the above mentioned inequalities for $C$-coconvex sets are similar to their analogues of convex bodies with opposite directions of inequalities. 
 
A direct application of the $L_p$ Minkowski inequality \eqref{LpMC} is the establishment of the uniqueness of solutions to the $L_p$ Minkowski problem for $C$-close sets raised by
Yang, Ye and Zhu \cite{YYZ22} and for $p=1$ by Schneider \cite{Sch18}. By the $L_p$ Minkwoski problem for $C$-close sets, we mean the characterization of the $L_p$ surface area measures of $C$-close sets. Progress on this problem include the existence of solutions to: the Minkowski problem (i.e., $p=1$) by Schneider \cite{Sch18, Sch21, Sch23} for finite and/or infinite measures, the Minkowski problem (i.e., $p=1$) by Zhang \cite{ZhangN24} for nonzero $\sigma$-finite Borel measures, the $L_p$ Minkowski problem  by Yang, Ye and Zhu \cite{YYZ22} when the finite measure is supported on compact subsets of $\Oc$ for  $0\neq p\in \mathbb{R}$, and the $L_p$ Minkowski problem for $0<p<1$ by Ai, Ye and Zhu \cite{AYZ} for finite Borel measures 
as well as the continuity of the solutions. A recent work by Schneider \cite{Sch232}  studied the  weighted Minkowski problem for finite measures.  Please refer to \cite{CW95, P80, U84} for similar works on the Minkowski type problems for unbounded convex hypersurfaces and \cite{HL21} for its $L_p$ version. 
Let us pause here to mention that the $L_p$ Minkowski problem for convex bodies, starting from the groundbreaking work of Lutwak  \cite{L93}, is one of the central objects in convex geometry, which has found fundamental applications in many areas, including analysis, affine geometry, partial differential equations, etc. Here we list a few contributions among others:  the centro-affine Minkowski problem \cite{CW06, JLZ16, Zhu152}, the $L_p$-Minkowski problem for polytopes \cite{HLYZ05, LYZ04, Zhu15, Zhu17}, the connections of the $L_p$ Minkowski problem with PDEs \cite{C06, HLX15, JLW15, LW13, U03}, and the establishment of the affine Sobolev inequalities and the Blaschke-Santal\'o inequalities \cite{CLYZ09, HS09, LYZ02, LZ97, Zhang992}.

The case $p=0$ is of particular interest. In this case, the $L_p$ Minkowski problems (for convex bodies or $C$-coconvex sets) become the log-Minkowski problems, aiming to characterize the cone-volume measures (for convex bodies or $C$-coconvex sets).  See \cite{BHZ16, BH16, BLYZ13,CLZ19, HL14, Sta02, Sta08, Zhu14} for the amazing contributions on the log-Minkowksi problem in the setting of convex bodies. The uniqueness of solutions to the log-Minkowksi problem for convex bodies is a major  problem, still quite open in convex geometry, see \cite{And99, BLYZ12, F74, Gage93, Sta03}. In the setting of the $C$-coconvex sets, Schneider in \cite{Sch18} established the existence of solutions to the  log-Minkowski problem, and raised an open problem regarding the uniqueness of solutions to the log-Minkowski problem. It has been discovered by Yang, Ye, and Zhu in \cite{YYZ22} that a  log-Minkowski inequality for $C$-coconvex sets can be established and hence the uniqueness of solutions to the log-Minkowski problem for $C$-coconvex sets can be obtained. This phenomenon shows that the settings of convex bodies and $C$-coconvex sets are quite different: when it is easy for convex bodies, it could be quite challenging for $C$-coconvex sets; and when it is challenging for convex bodies, it is usually easy for $C$-coconvex sets. 

The cone-volume measures (for both convex bodies and $C$-coconvex sets) have two different expressions. For example, the cone-volume measure of $C$-coconvex sets can be viewed as the $L_p$ surface area measure by letting $p=0$ in \eqref{Lp-measure-support}; it can also be viewed as $\widetilde{C}_q(\A, \cdot)$ for $q=n$. Here, $\widetilde{C}_q(\A, \cdot)$ is the $q$-th dual curvature measure of $C$-coconvex sets introduced by Li, Ye and Zhu in \cite{LYZ22}: 

\begin{equation*}
\widetilde{C}_q(\A, \eta)=
\left\{
\begin{array}{cc}
\frac{1}{n} \int_{\pmb{\alpha}^*_{\A}(\eta)} \rho_C(\A, v)^qdv,& q\neq 0, 
\\
\\
\int_{\pmb{\alpha}^*_{\A}(\eta)}  dv, &  q= 0,
\end{array}
\right.
\end{equation*} for each Borel set $\eta \subseteq \Omega_{C^\circ}$, where $\pmb{\alpha}^*_{\A}$ is the reverse radial Gauss map of $\A$. Under certain conditions, the $q$-th dual curvature measures of $C$-coconvex sets can be derived from a variationl formula of the $q$-th dual volume $\widetilde{V}_q(\cdot)$ in terms of the log-co-sum of $C$-coconvex sets. 

The above mentioned $q$-th dual curvature measure  can also be defined for $C$-compatible sets. Related dual Minkowski problem for $C$-compatible sets was posed in \cite{LYZ22}, which again asks whether a given measure defined on $\Oc$ can be the $q$-th dual curvature measure of some $C$-compatible sets. In \cite{LYZ22}, Li, Ye and Zhu were able to solve the above dual Minkowski problem for $C$-compatible sets when $q>0$ and the given measure is a nonzero finite Borel measure on $\Oc$. 
 
In this paper, we will study the $L_p$ dual Minkowski problem for $C$-compatible sets. Our motivation is the $L_p$ dual Minkowski problem for convex bodies initiated by Lutwak, Yang and Zhang \cite{LYZ18}. Remarkable contributions on the $L_p$ dual Minkowski probelm include \cite{BF19,  CHZ19, CL21, GLW22, HZ18, JWW21,  LLL22}, among others. The problem of interest is stated below. 

\vskip 2mm \noindent \textbf{The $L_p$ dual Minkowski problem for $C$-compatible sets:} {\em Given a nonzero finite Borel measure $\mu$ defined on $\Omega_{C^\circ}$ and real numbers $p, q$, under what conditions does there exist a $C$-compatible set $\A\in \Ln$ such that $\mu=\widetilde{C}_{p, q}(\A, \cdot)$, where} \begin{equation*}
\frac{\,d\widetilde{C}_{p, q}(\A, \cdot)}{\,d\widetilde{C}_{q}(\A, \cdot)} =(-h_C(\A, \cdot ))^{-p}\ ?  
\end{equation*} Hereafter, $\widetilde{C}_{p, q}(\A, \cdot)$ will be called the $(p, q)$-th dual curvature measure of $\A\in \Ln$. 
When $q=n,$  the $L_p$ dual Minkowski problem for $C$-compatible sets reduces to the  $L_p$ Minkowski problem \cite{Sch18, YYZ22}.  

The $L_p$ dual Minkowski problem for $C$-compatible sets reduces to the Monge-Amp\`{e}re type equations in certain circumstance. A typical example of such Monge-Amp\`{e}re type equations is  
\begin{equation*}
\left\{  \begin{array}{l}
(-h(u))^{1-p} \det\big(\bar{\nabla}^2 h(u)+h(u) I\big)=  f(u)  (h^2(u)+|\bar{\nabla} h(u)|^2)^{\frac{n-q}{2}} \ \ \mathrm{for}\ \  u \in \Omega_{C^\circ}, \\ \lim_{u \rightarrow \partial \Omega_{C^\circ}} h(u)=0,
\end{array}\right.
\end{equation*} 
where  $f: \Omega_{C^\circ} \rightarrow [0, \infty)$ is the density function of $\mu$ with respect to the spherical Lebesgue measure $\,du$, and $h:\Oc\rightarrow (-\infty, 0]$ is the unknown convex function. Here,  $I$ is the identity matrix, 
$\det B$ is the determinant of a matrix $B$,
$\bar{\nabla}$ and  $\bar{\nabla}^2$ are the gradient and, respectively,  Hessian operators with respect to an orthonormal frame on $S^{n-1}$, and $\partial \Oc$ is the boundary of $\Oc$. 

Observe that, $\widetilde{C}_{0, q}(\A, \cdot)=\widetilde{C}_{q}(\A, \cdot)$. Hence, the $L_p$ dual Minkowski problem becomes the dual Minkowski problem for $C$-compatible sets \cite{LYZ22} when $p=0$. The case $p=0$ and $q=0$ are closely related to the Alexandrov problem \cite{Ale42}.

Our main results are given  in Section \ref{sec4}, which establishes the existence of solutions to the $L_p$ dual Minkowski problem
for $C$-compatible sets. The first result is for $C$-full sets, whose statement focuses on $p\neq 0$ (and the case for $p=0$ can be seen in \cite{LYZ22}):

\begin{theorem}\label{cd1}
Let $p, q\in \mathbb{R},$ $p\neq 0$, and $\mu$ be a nonzero finite Borel measure defined on $\Omega_{C^\circ}$ whose support concentrates on a compact set $\omega \subset \Omega_{C^\circ}$. The following statements hold. 
 
\begin{itemize}
\item [(i)] If $q\neq 0$ and $p \neq q$, then  there exists a $C$-full set $\A$ such that $\mu=\widetilde{C}_{p, q} (\A, \cdot)$. 
\item [(ii)] If $q=0$, then there exists a $C$-full set $\A$ such that $\mu=J^*_{p} (\A, \cdot)$, where $J^*_p(\A, \cdot)=\widetilde{C}_{p, 0}(\A, \cdot)$. 
\end{itemize}
\end{theorem}

For $p, q\in\mathbb{R}$, a $C$-compatible set $\A \in \Ln$ is called a $(C, p, q)$-close set if 
$$\widetilde{C}_{p ,q}(\A, \Oc)=\int_{\Omega_{C^\circ}} (-h_C(\A, u))^{-p}d\widetilde{C}_q(\A, u)<\infty.$$ 
When $p=0$, the $(C, 0, q)$-close set is just the $(C, q)$-close set for $q \in \mathbb{R}$ proposed by Li, Ye and Zhu in \cite{LYZ22}. In particular, $(C, 0, n)$-close sets are the $C$-close sets in \cite{Sch18}. Applying the approximation method used in \cite{Sch18, Sch21, Sch23, Sch232}, Theorem \textcolor{red}{4.1} can be extended to a more general setting, namely, the $(C, p, q)$-close sets, which can be summarized as follows.

\begin{theorem}\label{Cpq-exis}
Let $p,q \in \mathbb{R}$ and $\mu$ be a nonzero finite Borel measure defined on $\Omega_{C^\circ}$. If  $p\leq 0$ and $p<q$, then there exists a $(C, p, q)$-close set $\A$ such that $\mu=\widetilde{C}_{p, q}(\A, \cdot)$.
\end{theorem}

This paper is organized as follows. In Section \ref{section:2}, we will provide some basic background. Section \ref{sec3} dedicates to the introduction of $(p ,q)$-th dual curvature measure and its related variational formula. The proofs of Theorems \ref{cd1} and  \ref{Cpq-exis} will be given in Section \ref{sec4}, while Section \ref{sec5} focuses on some results regarding the uniqueness of the solutions to the $L_p$ dual Minkowski problem for $C$-compatible sets.  

\section{Background and Preliminaries}\label{prepar} \label{section:2}

In this section, we will present some essential notations and background which are necessary for the presentation of this paper. For a comprehensive reference on these topics, we recommend books \cite{Sch14, Sch22} by Schneider, which serve as excellent resources.

Denote by $\overline{E}$,
$\text{int} E$, $\partial E$ and $E^c$, respectively, the closure,  the interior, the
boundary, and the complement of $E\subset \mathbb{R}^n$.   For $a\in \mathbb{R}$, the hyperplane with normal vector $u \in S^{n-1}$ is defined as
\begin{equation*}
H(u, a)=\{ x \in \mathbb{R}^n : x \cdot u=a\}.
\end{equation*} The upper and lower halfspaces can be given by
\begin{equation*}
H^{+}(u, a)=\{ x \in \mathbb{R}^n : x \cdot u \geq a\} \,\,\,\text{and}\,\,\,H^{-}(u, a)=\{ x \in \mathbb{R}^n : x \cdot u \leq a\}.
\end{equation*}

Recall that  $C\subset \mathbb{R}^n$ is a  pointed closed convex cone, if
$C=\{ \lambda x : \lambda \geq 0\,\, \text{and}\,\, x \in C\}$  is a closed and convex set such that $C$ has nonempty interior and 
$C \cap (-C)=\{o\}$
with $o$ being the origin of $\mathbb{R}^n$. Let $\Omega_C=\sphere \cap \text{int}C$ and 
$\Oc=\sphere \cap \text{int}C^{\circ}$, where $\sphere$ is the unit sphere in $\Rn$ and $C^{\circ}$ is polar cone of $C$ defined in \eqref{def-polar-cone}. The objects of interest are the $C$-compatible sets given in \eqref{def-1-1-2}, i.e.,  \begin{eqnarray*}
E=C\cap  \bigcap_{u\in \Oc}  \Big\{H_u^{-}:   E\subseteq H_u^-\Big\},
\end{eqnarray*} where $H_u^-$ denotes a halfspace with normal vector $u\in \sphere.$

Special $C$-compatible sets include $C$-close sets (i.e., the volume of $C\setminus E$ is finite), $C$-full sets (the set $C\setminus E$ is bounded), and the $C$-determined sets. Hereafter, for a compact set $\omega\subset \Oc$, a $C$-compatible set $\A$ is said to be $C$-determined by $\omega$, if 
\begin{equation}\label{Cdeter}
\A=C \cap \bigcap_{u \in \omega} H^-(u, h_C(\A, u)),
\end{equation} where $h_C(\A, \cdot)$ denotes the support function of $\A$ given in \eqref{support-1-A}.   Denote by $\mathscr{K}(C, \omega)$ the family of $C$-compatible sets that are $C$-determined by $\omega$. We say $\A$ a $C$-determined set if $\A\in \mathscr{K}(C, \omega)$ for some compact set $\omega\subset \Oc$. Clearly any $C$-determined set must be $C$-full. Note that for $E\in \Ln$, the hyperplane $H(u, h_C(E, u))$ is called the supporting hyperplane of $E$ with unit outer normal vector $u\in \Oc.$ Combining with the $p$-co-sum and the volume, Yang, Ye and Zhu \cite{YYZ22}  derived a variational formula using the Wulff shape, and subsequently introduced the notion of the $L_p$ surface area measure of a $C$-determined set $\A \in \mathscr{K}(C, \omega)$ when $0\neq p\in \mathbb{R}$ (see its definition in \eqref{Lp-measure-support}). We will derive a variational formula for the $q$-th dual volume in terms of the  $p$-co-sum of $C$-determined sets in Theorem \ref{varpq}.  

The convex hull and polarities are fundamental in convex geometry. In the case of $C$-compatible sets, such notions have been given by Li, Ye and Zhu in \cite{LYZ22}. For a nonempty set $o\notin E\subset C$, its closed convex hull with respect to $C$ is defined by 
\begin{align*}
\text{conv}(E, C)=\bigcap \left\{  \widetilde{E} :  \widetilde{E}\,\, \text{is a $C$-close set such that}\,\, E \subset  \widetilde{E} \right\} =C\cap  \bigcap_{u\in \Oc}  \Big\{H_u^{-}:   E\subseteq H_u^-\Big\}. 
\end{align*}   
In view of \eqref{def-1-1-2}, one sees that, $o\notin E\subset C$ is $C$-compatible if $E=\text{conv}(E, C).$  The copolar set of $o\notin E\subset C$ is defined  in \cite{LYZ22} (see also \cite{SSKzoo, R17, XLG23} for similar definitions) by 
\begin{equation*} \label{copolar-11}
E_C^{\diamond}=\{y \in C^\circ: x \cdot y \leq -1\,\,\,\text{for all}\,\,\,x \in E\} \subset C^{\circ}.\end{equation*} In particular, the bipolar theorem holds, namely, $(E_C^{\diamond})_{C^{\circ}}^{\diamond}=E$ if $E$ is a $C$-compatible set. Moreover, for any $C$-compatible set $E,$ one has   $$h_C(E, u)\cdot \rho_{C^\circ}(E_C^{\diamond}, u)=-1\,\,\, \text{for}\,\,\,u\in \Oc, $$  where $h_C(E, \cdot)$ and $\rho_C(E, \cdot)$ are the support and radial functions defined in \eqref{support-1-A} and \eqref{radial-1-A}. 

Recall that $\Ln$ denotes  the collection of $C$-compatible sets. For each $\A\in \Ln$, one can define $\nu_{\A}: \partial \A \cap \text{int} C \rightarrow S^{n-1}$, the Gauss map of $\A$, by: for each $F \subseteq \partial \A\cap \text{int} C$, 
\begin{equation*}
\nu_{\A}(F)=\{u\in S^{n-1} : F\cap H \left( u, h_C(\A, u)\right)\neq \emptyset \},
\end{equation*}
where $ H \left( u, h_C(\A, u)\right)$ is the supporting hyperplane of $\A$ at the direction $u$. By $\nu^{-1}_{\A}$, we mean the reverse Gauss map of $\A$. Let ${S}_{n-1}(\A, \cdot)$ denote the surface area measure of $\A\in \Ln$, and then, for any Borel set $\eta \subseteq \Omega_{C^\circ}$, one has 
\begin{equation*}
S_{n-1}(\A, \eta)=\mathcal{H}^{n-1} (\nu^{-1}_{\A}(\eta)),
\end{equation*}
where $\mathcal{H}^{k}$ is the $k$-dimensional Hausdorff measure. When $\A\in \Ln$ is a $C$-close set, Schneider in  \cite{Sch18} formulated the volume of $A=C \,\backslash \A$  as follows:
\begin{equation*}
V_n(A)=\frac{1}{n}\int_{\Omega_{C^\circ}} \overline{h}_C(A, u)d\overline{S}_{n-1}(A, u)=-\frac{1}{n}\int_{\Omega_{C^\circ}} h_C(\A, u)dS_{n-1}(\A, u),
\end{equation*}
where $\overline{S}_{n-1}(A, \cdot)=S_{n-1}(\A, \cdot).$

Recall that the $p$-co-sum of $C$-coconvex sets is defined in \eqref{co-p-sum-def-1} and \eqref{co-p-sum-def-2}. Note that formula \eqref{co-p-sum-def-2} can be also used to define the $p$-co-sum of $C$-compatible sets. A slight modification of \eqref{co-p-sum-def-2} can be used to define the log-co-sum of $C$-compatible sets (i.e., the case $p=0$). Indeed, for $\tau \in [0, 1]$ and $\A_1, \A_2\in \Ln$, let $(1-\tau)A_1 \oplus_0 \tau A_2$ denote the log-co-sum of $A_1$ and $A_2$ with respect to $\tau$, which is defined by its support function given by, for $u \in \Omega_{C^\circ}$, 
\begin{equation*}
\overline{h}_C((1-\tau) A_1 \oplus_0 \tau A_2, u)=(\overline{h}_C(A_1, u))^{1-\tau}(\overline{h}_C(A_2, u))^\tau= (-h_C(\A_1, u))^{1-\tau}(-h_C(\A_2, u))^\tau. 
\end{equation*}
That is, the log-co-sum of $C$-coconvex sets $A_1$ and $A_2$ is given by:
\begin{equation*}
C\,\backslash ((1-\tau)A_1 \oplus_0 \tau A_2)=C \cap \bigcap_{u\in \Oc} \big\{ x \in \mathbb{R}^n: x\cdot u\leq -\overline{h}_C((1-\tau)A_1 \oplus_0 \tau A_2, u)\big\}.
\end{equation*}
  
We shall use the following fact: there exists a fixed vector $\xi \in \Omega_{C}$, such that $x \cdot \xi>0$ for all $x \in C\, \backslash\, \{o\}$.
Let $C_t=C \cap H^{-}_{t}$, where $H^{-}_{t}=H^{-}(\xi, t)$ for $t\in \mathbb{R}$.
It is easy to see that $C_t$ is bounded for $t>0$. We also use the following convergence for $C$-compatible sets, see \cite{LYZ22, Sch18}. For convenience, let  $\mathbb{N}_0=\mathbb{N} \cup \{0\}$.  
\begin{definition}\label{convergence}
For a sequence $\{\A_i\}_{i \in \mathbb{N}_0} \subset \Ln$, if there exists $t_0>0$ such that $\A_i \cap C_{t_0} \neq \emptyset$ for all $i \in \mathbb{N}$, and for all $t>t_0$, 
\begin{equation*}
\A_i \cap C_{t} \rightarrow \A_0\cap C_{t} \,\,\, \text{as}\,\,\, i\rightarrow \infty
\end{equation*}
in the Hausdorff metric, then we say $\A_i $ converges to $\A_0$ as $i\rightarrow \infty$, written by $\A_i \rightarrow  \A_0$ as $i\rightarrow \infty$.
\end{definition}

The  effective boundary of a $C$-compatible set $\A$ may be defined by $\partial_e \A=\partial \A \cap \partial A$. The effective radial direction of $\A$ is given by 
$$\Omega_C^e=\Big\{\frac{x}{|x|}: x\in \partial_e \A\Big\}.$$ That is, $\Omega_C^e\subseteq C\cap S^{n-1}$ satisfies that 
$\rho_C(\A,v)\in (0, \infty)$ if $v\in \Omega_C^e$.
Similarly, we will have the effective range of the unit normal vectors of $\partial_e \A$, which is of the form $\Omega_{C^\circ}^e=\Omega_{C^\circ} \cup \mathcal{N}_{\A}$, where  either $\mathcal{N}_{\A} =\emptyset$ or $\mathcal{N}_{\A} \subset (C^{\circ}\cap S^{n-1}) \backslash\, \Omega_{C^\circ}$. Clearly, $\mathcal{N}_{\A}$ has its  spherical Lebesgue measure equal to $0$.
Notice that $\Omega_C\subseteq \Omega_C^e$ and the spherical Lebesgue measure of  $\Omega_C^e\setminus \Omega_C$ is 0.  
Thus, the $q$-th dual volume of $\A \in \Ln$ for $0 \neq q \in \mathbb{R}$ defined in \cite{LYZ22} can be expressed as
\begin{equation}\label{qthdual_Volume}
\widetilde{V}_q(\A)=\frac{1}{n} \int_{\Omega_{C}^e} \rho_C(\A, v)^q dv=\frac{1}{n} \int_{\Omega_{C}} \rho_C(\A, v)^q dv
\end{equation}
given that the above integral exists and is finite. 
Note that when $q=n$, $\widetilde{V}_n(\A)=V_n(C\,\backslash \A)$ for a $C$-close set $\A$.
Similarly, the dual entropy of $C$-compatible set $\A$ defined in \cite{LYZ22} is given by
\begin{equation}\label{dual_entropy}
\widetilde{\mathbb{E}}(\A)= \int_{\Omega_C} \log\big(\rho_C(\A, v)\big) \,dv
\end{equation}
on the condition that \eqref{dual_entropy} exists and is finite.
 
\section{The \texorpdfstring{$(p ,q)$}{}-th dual curvature measure and variational formula} \label{sec3}

This section is devoted to precisely defining the $(p, q)$-th dual curvature measure of $C$-compatible sets and establishing the associated variational formula for $C$-determined sets. Prior to the study on these topics, certain preliminary work is necessary.  To ensure coherence and build upon existing research, we adopt the methodology employed by Li, Ye and Zhu in their recent work \cite{LYZ22}, and more details can be found therein.

Let $\A \in \Ln$.
Denote by $\pmb{\alpha}_{\A}$ the radial Gauss map of $\A$, which is formulated by:  for $\vartheta \subseteq \Omega^e_C$,
\begin{equation}\label{radial&Gauss}
\pmb{\alpha}_{\A}(\vartheta)=\{ u \in \Omega^e_{C^\circ} : r_{\A}(v) \in H(u, h_C(\A, u))\,\,\text{for some}\,\, v \in\vartheta \},
\end{equation}
where $ r_{\A}: \Omega^e_C \rightarrow  \partial_e \A$ is the radial map of $\A$ given by $r_{\A}(v)=\rho_C(\A, v)v \in \partial_e \A$ for $v \in \Omega^e_C$.
The reverse radial Gauss map of $\A$, denoted by $\pmb{\alpha}^*_{\A}$, is given by:  for $\eta \subseteq \Omega^e_{C^\circ}$,
\begin{equation}\label{reverse&radial&Gauss}
\pmb{\alpha}^*_{\A}(\eta)=\{ v \in \Omega^e_{C} : r_{\A}(v) \in H(u, h_C(\A, u))\,\,\text{for some}\,\, u \in\eta\}.
\end{equation}
From \eqref{radial&Gauss} and \eqref{reverse&radial&Gauss}, one sees that,  for $v \in \Omega^e_C$ and $\eta\subseteq \Omega^e_{C^\circ}$,
\begin{equation*}
v \in \pmb{\alpha}^*_{\A}(\eta) \Longleftrightarrow \pmb{\alpha}_{\A}(\{v\}) \cap \eta= \pmb{\alpha}_{\A}(v) \cap \eta \neq \emptyset.
\end{equation*}
Particularly, if $\eta$ contains only one element $u \in \Omega^e_{C^\circ}$, one has
\begin{equation}\label{singleton}
v \in \pmb{\alpha}^*_{\A}(u) \Longleftrightarrow u \in \pmb{\alpha}_{\A}(v).
\end{equation}

Recall the definition of  $q$-th dual curvature measure of $\A$ for $q\in \R$ given in \cite[Definition 4.3 and Definition 8.1]{LYZ22}. 
\begin{definition}\label{q-dual-cur}
Let $0\neq q \in \mathbb{R}$ and $\A \in \Ln$. If $\widetilde{V}_q(\A)<\infty$, define the $q$-th dual curvature measure of $\A$, denoted by $\widetilde{C}_{q}(\A, \cdot)$,  as: for each Borel set $\eta \subseteq \Omega_{C^\circ}$,
\begin{equation}\label{qdual}
\widetilde{C}_q(\A, \eta)=\frac{1}{n} \int_{{\alpha}^*_{\A}(\eta)} \rho_C(\A, v)^qdv.
\end{equation}
When $q=0$, define
\begin{equation}\label{0dual}
\widetilde{C}_0(\A, \eta)= \int_{\pmb{\alpha}^*_{\A}(\eta)}  dv.
\end{equation}
\end{definition}

Based on Definition \ref{q-dual-cur},  we can define the $(p, q)$-th dual curvature measures of $C$-compatible sets.  Chen and Tu, independently, also come up with the same definition of  the $(p, q)$-th dual curvature measures of $C$-compatible sets in their recent work \cite{CT24}.   

\begin{definition}\label{pqdual}
Let $p, q \in \mathbb{R}$ and $\A \in \Ln$. Define the $(p, q)$-th dual curvature measure of $\A$, denoted by $\widetilde{C}_{p ,q}(\A, \cdot)$, by: for each Borel set $\eta \subseteq \Omega_{C^\circ}$,
\begin{equation}\label{pq-dual}
\widetilde{C}_{p, q}(\A, \eta)=\int_{\eta} (-h_C(\A, u))^{-p} d\widetilde{C}_q (\A, u).
\end{equation}
\end{definition} 

Let $\A\in \Ln$. For each $u\in \Oc$, one has $h_C(\A, u)\in (-\infty, 0)$. Hence, the $(p, q)$-th dual curvature measure $\widetilde{C}_{p, q}(\A, \cdot)$ is well-defined but may be infinite. Moreover, the following holds on $\Oc$: 
$$ \frac{d \widetilde{C}_{p, q}(\A, \cdot)}{d\widetilde{C}_q (\A, \cdot)}=(-h_C(\A, \cdot))^{-p}.$$ In particular,  $\widetilde{C}_{p, q}(\A, \cdot)=\widetilde{C}_{ q}(\A, \cdot)$ when $p=0$. 
Notice that $\widetilde{C}_0(\A, \cdot)$ is usually written by $J^*(\A, \cdot)$ for $q=0$ (see \cite{LYZ22}). Thus,
for $p\in \mathbb{R}$,
$\widetilde{C}_{p, 0}(\A, \cdot)$ is usually written as
$d\widetilde{C}_{p, 0}(\A, \cdot)=dJ_p^*(\A, \cdot)=(-h_C(\A, \cdot))^{-p} dJ^*(\A, \cdot)$. 
For $\lambda>0$,  it follows from \eqref{qdual}, \eqref{0dual} and \eqref{pq-dual} that 
\begin{align} \label{homCpq}
\widetilde{C}_{p, q}(\lambda \A, \cdot)&=\lambda^{q-p}\widetilde{C}_{p, q}( \A, \cdot)\ \ \  \text{for} \ \ q \neq 0, \\  
J^*_p(\lambda \A, \cdot)&=\lambda^{-p}J^*_p( \A, \cdot)\ \ \ \  \ \  \text{for} \ \ q = 0. \notag
\end{align}
For $A=C\,\backslash \A$, we simply let $\widetilde{C}_{p, q}(A, \cdot)=\widetilde{C}_{p, q}(\A, \cdot)$.

We are interested in a special family of $C$-compatible sets, called $(C, p, q)$-close sets.  
 
\begin{definition}\label{CpqC} A $C$-compatible set $\A \in \Ln$ is said to be  $(C, p, q)$-close for $p, q \in \mathbb{R}$, if 
\begin{equation}\label{Cpqclose}
\widetilde{C}_{p, q}(\A, \Oc)=\int_{\Omega_{C^\circ}} (-h_C(\A, u))^{-p}d\widetilde{C}_q(\A, u)<\infty.
\end{equation} If $\A \in \Ln$ is $(C, p, q)$-close, then $A=C\,\backslash \A$ is called a $(C, p, q)$-coconvex set. 
\end{definition} Note that, a $(C, 0, q)$-close set is just a $(C, q)$-close set proposed by Li, Ye and Zhu in \cite{LYZ22}, which contains the $C$-close sets as special cases (corresponding to $q=n$).   
It can be checked from \eqref{qdual}, \eqref{0dual},  the boundedness of its support function that every $(C, q)$-close set must be a $(C, p, q)$-close set for all $p< 0$,  while every $(C, p, q)$-close set for $p>0$   must be a $(C, q)$-close set. A $C$-full set must be a $(C, p, q)$-close set for all $p\leq 0$. The $C$-determined sets are special, because each $\A \in \mathscr{K}(C, \omega)$ for some compact subset $\omega \subset \Omega_{C^\circ}$ must be $(C, p, q)$-close  for any $p, q\in \R$. To this end, notice that 
$\pmb{\alpha}^*_{\A}$ maps each Borel set $\eta\subset \Oc$ 
disjoint with $\omega$ to the boundary of $C$, and hence $\pmb{\alpha}^*_{\A}(\eta)$ has its spherical measure equal to $0$. That says, the measure $\widetilde{C}_q (\A, \cdot)$ is concentrated on $\omega$ and $\A\in  \mathscr{K}(C, \omega)$ must be $(C, p, q)$-close  for any $p, q\in \R$, due to  \eqref{Cpqclose}. 

By \cite[(4.11), (4.12) and Section 8]{LYZ22}, for $ q \in \mathbb{R}$ and each bounded Borel function $f:\Omega_{C^\circ} \rightarrow \mathbb{R}$, the following hold: 
\begin{itemize}
\item if $q\neq 0$,
\begin{align} 
\notag
\int_{\Omega_{C^\circ}}f(u)d\widetilde{C}_q(\A, u)&=\frac{1}{n}\int_{\Omega_{C}} f(\alpha_{\A}(v))\rho_C(\A, v)^qdv,\\ \notag 
\int_{\Omega_{C^\circ}}f(u)d\widetilde{C}_q(\A, u)&=-\frac{1}{n}\int_{\partial \A \cap (\text{int}C)} f(\nu_\A (x)) \cdot (x \cdot \nu_\A (x)) \cdot |x|^{q-n} d\mathscr{H}^{n-1} (x);
\end{align}

\item if $q=0$,
\begin{align}
\label{Borel&q0}
\int_{\Omega_{C^\circ}}f(u)dJ^*(\A, u)&=\int_{\Omega_{C}} f(\alpha_{\A}(v))dv, \\  
\int_{\Omega_{C^\circ}}f(u)dJ^*(\A, u)&=-\int_{\partial \A \cap (\text{int}C)} f(\nu_\A (x)) \cdot (x \cdot \nu_\A (x)) \cdot |x|^{-n} d\mathscr{H}^{n-1} (x) \notag.
\end{align}
\end{itemize} 

Consequently, the following lemma can be obtained for $(C, p, q)$-close sets.
\begin{lemma}
Let $p, q \in \mathbb{R}$ and $\A \in \Ln$ be a $(C, p, q)$-close set. Then for each Borel set $\eta \subseteq \Omega_{C^\circ}$ and each bounded Borel function $g: \Omega_{C^\circ} \rightarrow \mathbb{R}$, the following hold. 
\begin{itemize}
\item[(i)] If $q \neq 0$, then 
\begin{align}\label{borel&q}
\int_{\Omega_{C^\circ}} g(u)d\widetilde{C}_{p ,q}(\A, u)&=\frac{1}{n} \int_{\Omega_{C}} g(\alpha_{\A}(v)) (-h_C(\A, \alpha_{\A}(v)))^{-p}\rho_C(\A, v)^qdv, \\  \label{borel&qL}
\int_{\Omega_{C^\circ}} g(u)d\widetilde{C}_{p ,q}(\A, u)&=\frac{1}{n}\int_{\partial \A \cap (\text{int}C)} g(\nu_\A (x)) \cdot [-(x \cdot \nu_\A (x))]^{1-p} \cdot |x|^{q-n} d\mathscr{H}^{n-1} (x). 
\end{align}

\item[(ii)] If $q= 0$, then
\begin{align}\label{borel&0}
\int_{\Omega_{C^\circ}} g(u)dJ^*_{p}(\A, u)&=\int_{\Omega_{C}} g(\alpha_{\A}(v)) (-h_C(\A, \alpha_{\A}(v)))^{-p}dv,\\ \label{borel&0L}
\int_{\Omega_{C^\circ}} g(u)dJ^*_{p}(\A, u)&=\int_{\partial \A \cap (\text{int}C)} g(\nu_\A (x)) \cdot [-(x \cdot \nu_\A (x))]^{1-p} \cdot |x|^{-n} d\mathscr{H}^{n-1} (x).
\end{align}
\end{itemize}
\end{lemma}

\begin{proof} 
Formulas \eqref{borel&q} and \eqref{borel&0} can be obtained through a standard argument based on simple functions. Formulas \eqref{borel&qL} and \eqref{borel&0L} can be obtained by following variable change: $u=\frac{x}{|x|}=r_{\A}^{-1}(x)$, whose Jacobian is $-\big(x\cdot \nu_{\A}(x)\big) \cdot |x|^{-n}$ \cite[page 22]{LYZ22}.   
\end{proof} 

The next proposition, which follows similar lines to \cite[Propsition 4.3]{LYZ18}, exhibits the integral formulation of the $(p, q)$-th dual curvature measure of a $(C, p, q)$-close set $\A$ for $p, q \in \mathbb{R}$.
\begin{prop}\label{integral}
Let $p, q \in \mathbb{R}$  and $\A \in \Ln$ be a $(C, p, q)$-close set. The following hold for each Borel set $\eta \subseteq \Omega_{C^\circ}$: 

\begin{itemize}
\item [(i)] if $q\neq 0$, one has
\begin{equation}
\widetilde{C}_{p, q}(\A, \eta)=\frac{1}{n} \int_{{\pmb{\alpha}}^*_{\A}(\eta)} (-h_C(\A, {\pmb{\alpha}}_{\A}(v)) )^{-p}\rho_C(\A, v)^qdv;\label{need to prove-1}
\end{equation}

\item [(ii)] if $q= 0$, one has
\begin{equation*}
J^*_p(\A, \eta)= \int_{{\pmb{\alpha}}^*_{\A}(\eta)} (-h_C(\A, {\pmb{\alpha}}_{\A}(v)) )^{-p}dv.
\end{equation*}
\end{itemize}
\end{prop}

\begin{proof} We only need to verify the first statement since the other one follows along the same lines. Indeed, formula \eqref{need to prove-1} follows from \eqref{singleton}, \eqref{borel&q}, and Definition \ref{pqdual}. That is, for each Borel set $\eta \subseteq \Omega_{C^\circ}$, 
\begin{align*}
\widetilde{C}_{p, q}(\A, \eta)
&=\int_{\Omega_{C^\circ}} \mathbbm{1}_\eta(u)d\widetilde{C}_{p, q}(\A, u)\\
&=\int_{\Omega_{C^\circ}} \mathbbm{1}_\eta(u)(-h_C(\A, u))^{-p}d\widetilde{C}_q(\A, u)\\
&=\frac{1}{n}\int_{\Omega_C} \mathbbm{1}_\eta({\pmb{\alpha}}_{\A}(v))(-h_C(\A, {\pmb{\alpha}}_{\A}(v)))^{-p}\rho_C(\A, v)^qdv\\
&=\frac{1}{n}\int_{\Omega_C} \mathbbm{1}_{{\pmb{\alpha}}^*_{\A}(\eta)}(v)(-h_C(\A, {\pmb{\alpha}}_{\A}(v)))^{-p}\rho_C(\A, v)^qdv\\
&=\frac{1}{n}\int_{{\pmb{\alpha}}^*_{\A}(\eta)} (-h_C(\A, {\pmb{\alpha}}_{\A}(v)))^{-p}\rho_C(\A, v)^qdv.
\end{align*} This completes the proof. 
\end{proof}

The following Minkowski type problem is of interest. A special case for $C$-close sets has been, independently, proposed by Chen and Tu in their recent work \cite[Problem 2.2]{CT24}. 

\begin{problem}[The $L_p$ dual Minkowski problem for $C$-compatible sets]\label{qthdual?}
For $p, q \in \mathbb{R}$ and a nonzero finite Borel measure $\mu$ defined on  $\Omega_{C^\circ}$, under what conditions does there exist a $C$-compatible set $\A \in \Ln$  such that $\mu=\widetilde{C}_{p, q} (\A, \cdot)?$
\end{problem}
When $q=n$, it reduces to the $L_p$ Minkowski problem for $C$-compatible sets \cite{Sch18, Sch21, YYZ22}. 
When $p=0$, it reduces to the dual Minkowski problem previously proposed by Li, Ye, and Zhu \cite{LYZ22}. Therefore, our attention will focus on the case where $0\neq p\in \mathbb{R}$ and $q \in \mathbb{R}$. 

If the density $f$ has enough smoothness and let $\,d\mu= f\,du$ with $f: \Omega_{C^\circ} \rightarrow [0, \infty)$, to solve the above $L_p$ dual Minkowski problem for $C$-compatible sets is equivalent to find a convex solution $h: \Omega_{C^\circ} \rightarrow (-\infty, 0)$ to the following Monge-Amp\`{e}re type equation 
\begin{equation}\label{MA-unbounded-1}
\big(\!-h(u)\big)^{1-p} \det\big(\bar{\nabla}^2 h(u)+h(u) I\big)=    f(u)  (h^2(u)+|\bar{\nabla} h(u)|^2)^{\frac{n-q}{2}},
\end{equation}
where $\bar{\nabla}$ and $\bar{\nabla}^2$ denote the gradient and Hessian operators with respect to an orthonormal frame on $S^{n-1}$,   $I$ is the identity matrix and $|x|$ is the Euclidean norm of $x\in \mathbb{R}^n$.  Independently, a special case for $C$-close sets has been proposed by Chen and Tu  in their recent work \cite[(1.2)]{CT24}. Moreover, they studied related existence, regularity, and uniqueness of the above Monge-Amp\`{e}re type equation \eqref{MA-unbounded-1} for $p\geq 1$ in \cite[Theorem 1.2]{CT24}. An existence and optimal global H\"{o}lder regularity in the case $p<1$ and $q\geq n$ were studied in \cite[Theorem 1.3]{CT24}. Our goal in Section  \ref{sec4} of the present paper is to find weak solutions to the above Monge-Amp\`{e}re type equation \eqref{MA-unbounded-1} for $p\leq 0$ and $p<q$, and also to establish the uniqueness of solutions to
Problem \ref{qthdual?} (i.e., the $L_p$ dual Minkowski problem for $C$-compatible sets) for $p\leq q$.  

For a compact subset $\omega \subset \Omega_{C^\circ}$, denote by $C^{+}(\omega)$ and $C(\omega)$ the set of positive continuous functions on $\omega$ and the set of continuous functions on $\omega$, respectively. Let $f_0 \in C^{+}(\omega)$, $g \in C(\omega)$ and $\epsilon >0$ be sufficiently small. Define $f_t \in C^{+}(\omega)$ for $t \in (-\epsilon, \epsilon)$ and $u \in \omega$ by
\begin{equation}\label{ft}
\log f_t(u)=\log f_0(u)+tg(u)+o(t, u),
\end{equation}
where the function $o(t, \cdot): \omega \rightarrow \mathbb{R}$ is continuous and $\lim_{t \rightarrow 0} o(t, u)/t=0$ uniformly on $\omega$. By $[f_t]=[C, \omega, f_t]$, we mean the Wulff shape associated with $f_t$: 
\begin{equation}\label{wul}
[f_t]=C \,\cap\,\, \bigcap_{u \in \omega} H^{-} (u, -f_t(u)).
\end{equation}

The following results have been proved in \cite[Theorem 5.1 and Theorem 8.1]{LYZ22}.
Recall that $\widetilde{V}_q(\A)$ and $\mathbb{E}(\A)$ are the $q$-th dual volume and the dual entropy of $\A \in \Ln$ defined by \eqref{qthdual_Volume} and \eqref{dual_entropy}, respectively.

\begin{lemma} \label{var}
Let $q \in \mathbb{R}$ and $\omega \subset \Omega_{C^\circ}$ be a compact set. If
$f_0, g, f_t, [f_t]$ are given in \eqref{ft} and \eqref{wul}, then 
\begin{align*}
\lim_{t \rightarrow 0} \frac{\widetilde{V}_q ([f_t])-\widetilde{V}_q ([f_0])}{t}&=q \int_{\omega} g(u) d\widetilde{C}_q ([f_0], u)\ \ \text{for}\ \ q\neq 0,\\ 
\lim_{t\rightarrow 0}\frac{\widetilde{\mathbb{E}}([f_t])-\widetilde{\mathbb{E}}([f_0])}{t}&=
\int_{\omega}g(u)\,dJ^*([f_0], u)
\ \ \text{for}\ \ q= 0.
\end{align*}
\end{lemma}

We now prove the variational formula for the $(p, q)$-th dual curvature measure of $C$-determined sets for $p, q\in \R$ with $p\neq 0$. To this end, 
let $\omega\subset \Omega_{C^\circ}$ be a compact set and $f_t$ be defined  by: for $u\in \omega$, 
\begin{equation}\label{ftt}
f_t(u)=(f_0(u)^p+tg(u))^\frac{1}{p},
\end{equation}
where $g \in C(\omega)$ and $t \in (-\epsilon, \epsilon)$ for some $\epsilon>0$ sufficiently small. As $f_0 \in C^{+}(\omega)$, one can find  a constant  $\epsilon_0$  such that
\begin{equation*}
0<\epsilon_0<\frac{\min_{u \in \omega} f_0(u)^p}{\max_{u \in \omega}|g(u)|+1}.
\end{equation*}
Hence, it can be verified that $f_t \in C^+(\omega)$  for all $t \in (-\epsilon, \epsilon)\subseteq (-\epsilon_0, \epsilon_0)$. Since both $g$ and $f_0$ are continuous functions on $\omega$ and $f_0$ is positive,  by the chain rule, it follows that
\begin{equation*}
\lim_{t\rightarrow 0} \frac{\log f_t(u)-\log f_0(u)}{t}=\frac{1}{p} \cdot  \frac{g(u)}{f_0(u)^{p}} \in C(\omega) 
\end{equation*}
uniformly on $\omega$. Then, we can rewrite $f_t$ in the form of \eqref{ft} by the Taylor extension formula as
\begin{equation}\label{fzthm3.1-1}
\log f_t(u)=\log f_0(u)+\left(\frac{1}{p} \cdot \frac{g(u)}{f_0(u)^p}\right)t +o(t, u),
\end{equation} where $o(t, \cdot)$ is continuous on $\omega$ and $\lim_{t \rightarrow 0} o(t, \cdot)/t=0$ uniformly on $\omega$. Applying Lemma \ref{var} to \eqref{fzthm3.1-1}, one can easily get the following variational formula for the $(p, q)$-th dual curvature measure. The case for $p=0$ is Lemma \ref{var} itself.

\begin{theorem}\label{varpq}
Let $p, q \in \mathbb{R}$, and $\omega\subset \Omega_{C^\circ}$ be a compact subset.  Let $f_0, g, f_t$ be given by \eqref{ftt} for $p\neq 0$ and by \eqref{ft} for $p=0$. Let $[f_t]$ be defined by \eqref{wul}. Then, the following statements hold: 
\begin{itemize}
\item [(i)] if $q\neq 0$, one has
\begin{align*}
\lim_{t \rightarrow 0} \frac{\widetilde{V}_q ([f_t])-\widetilde{V}_q ([f_0])}{t}=
\left\{
\begin{array}{ll}
\frac{q}{p} \int_{\omega} g(u) f_0(u)^{-p}d\widetilde{C}_{q} ([f_0], u),\ \ &\text{for} \ \ p\neq 0,\\ \\
q \int_{\omega}  g(u) d\widetilde{C}_{q} ([f_0], u),\ \ &\text{for} \ \ p= 0;
\end{array}
\right.
\end{align*}
\item[(ii)] if $q= 0$, one has
\begin{align*}
\lim_{t \rightarrow 0} \frac{\widetilde{\mathbb{E}} ([f_t])-\widetilde{\mathbb{E}} ([f_0])}{t}=\left\{
\begin{array}{ll}
\frac{1}{p} \int_{\omega} g(u) f_0(u)^{-p} dJ^* ([f_0], u),\ \ &\text{for} \ \ p\neq 0,\\ \\
\int_{\omega} g(u) dJ^* ([f_0], u),\ \ &\text{for} \ \ p= 0.
\end{array}
\right.
\end{align*}
\end{itemize}
\end{theorem}

\section{Existence of solutions to the \texorpdfstring{$L_p$}{} dual Minkowski problem for \texorpdfstring{$(C, p, q)$}{}-close sets}\label{sec4}

In this section, 
we will prove the existence of solutions to the $L_p$ dual Minkowski problem for $(C, p, q)$-close sets, under certain conditions on $p, q\in \R.$

\subsection{The \texorpdfstring{$L_p$}{} dual Minkowski problem for \texorpdfstring{$C$}{}-determined sets} In this subsection, we aim to solve the $L_p$ dual Minkowski problem for $C$-determined sets with concentration on $p\neq 0$. The case for $p=0$ has been studied in \cite{LYZ22}.

Let $0\neq p\in \mathbb{R}$. Consider a nonzero finite Borel measure $\mu$ defined on $\Omega_{C^\circ}$, with its support concentrated on a compact set $\omega \subset \Omega_{C^\circ}$. 
For $f \in C^+(\omega)$, let $$\|f\|_p=\left(\int _{\omega} f(u)^p\,d\mu(u)\right)^{\frac{1}{p}}.$$ Note that $\|\cdot\|_p$ defines a norm if $p\geq 1$. It is easily checked that, for all $p\neq 0$,  $\|f\|_p\leq \|g\|_p$ for any $f\leq g$. Moreover, for any $\lambda >0$, $\|\lambda f\|_p=\lambda \|f\|_p$. Define the functional  $\Phi: C^+(\omega)\rightarrow \R$ by
\begin{equation}\label{Phif}
\Phi(f)=\begin{cases}
 - \log \|f\|_p+\frac{1}{q}\log \widetilde{V}_{q}([f]), \ \ q\neq 0,\\ \\
-  \log \|f\|_p + \frac{1}{|\Omega_{C}|}\widetilde{\mathbb{E}}([f]),\ \ \ \ \ q= 0,  
\end{cases}
\end{equation}  
where $|\Omega_{C}|=\int_{\Omega_{C}}dv$. 
The expression \eqref{wul} implies that $[f] \in  \mathscr{K}(C, \omega)$, and thus, $\Phi$ is homogeneous of degree zero. To this end, for any $\lambda>0$, the fact that  $[\lambda f]=\lambda[f]$ yields $\widetilde{V}_{q}([\lambda f])=\lambda^{q}\widetilde{V}_{q}([f])$ by \eqref{qthdual_Volume}, and $\widetilde{\mathbb{E}}( [\lambda f])=\widetilde{\mathbb{E}}([f])+ |\Omega_{C}|\cdot \log \lambda$ following from \eqref{dual_entropy}.
It follows from \eqref{Phif} that, for $q\neq 0$,  
\begin{align}\label{h0}
\Phi(\lambda f) \notag
&=-  \log \| \lambda f\|_p+\frac{1}{q}\log \widetilde{V}_{q}([\lambda f])\\ \notag
&=- \log \|f\|_p- \log \lambda +\frac{1}{q}\log \widetilde{V}_{q}([f])+\frac{1}{q}\log\lambda^{q} =\Phi(f),
\end{align} and the case for $q=0$ follows along the same lines. 

For $Q \in \mathscr{K}(C, \omega)$, we define $\Phi(Q)$ as
\begin{equation}\label{PhiQ}
\Phi(Q)=\Phi(-h_C(Q, \cdot))=
\begin{cases}
-  \log \|-h_C(Q, \cdot)\|_p +\frac{1}{q}\log \widetilde{V}_{q}(Q), \ \  q\neq 0,\\ \\
-  \log \|-h_C(Q,\cdot)\|_p  +\frac{1}{|\Omega_{C}|}\widetilde{\mathbb{E}}(Q),\ \ \ \ q=0.
\end{cases}
\end{equation}
  
This leads us to consider the following optimization problems: 
$$\Upsilon_f=\inf \{  \Phi(f) : f \in C^{+}(\omega) \}\,\,\,\,\, \text{and}\,\,\,\,\,\,\Upsilon_Q=\inf \{  \Phi(Q) : Q \in \mathscr{K}(C, \omega) \}.$$ Both optimization problems are well-defined as the functional $\Phi(\cdot)$ is  homogeneous of degree $0$. 

The subsequent lemma, resulting from the combination of \cite[Lemma 6.1 and Lemma 8.1]{LYZ22}, asserts the properties of $C$-full sets $\A \in \mathscr{K}(C, \omega)$. 

\begin{lemma} \label{Cqb&t2}
Let $q \in \mathbb{R}$. The following statements hold. 

\begin{itemize}
\item [(i)]  Let $q\neq 0$. There is a constant $t_0>0$, only depending on $\omega \subset \Omega_{C^\circ}$ and $C$  with the following property: if $\A\in \mathcal{K}(C, \omega)$ and $\widetilde{V}_q(C\,\backslash \A)=1$, then   $C \cap H_{t_0} \subset \A$.   
\item [(ii)] There is a constant $t_1>0$, only depending on $\omega \subset \Omega_{C^\circ}$ and $C$ with the following property:   if $\A\in \mathcal{K}(C, \omega)$ and  $\widetilde{\mathbb{E}} (C\setminus \A)=0$, then $C \cap H_{t_1} \subset \A$.
\end{itemize}
\end{lemma}

The following lemma, which summarizes \cite[Lemma 5 and Lemma 6]{Sch18} is needed. Recall that $\mathbb{N}_0=\mathbb{N}\cup \{0\}$.

\begin{lemma}\label{L56}
Let $\omega\subset \Oc$ be a compact set. The following statements hold true.
\begin{itemize}
\item [(i)] If $f_i\in C^+(\omega)$ for all $i\in \mathbb{N}_0$ such that  $f_i\rightarrow f_0$ uniformly on $\omega$ as $i\rightarrow \infty$, then  $[C, \omega, f_j]\rightarrow [C, \omega, f]$.
\item [(ii)] If $\{\A_j\}_{j\in \mathbb{N}} \subset \mathcal{K}(C, \omega)$ such that $\A_j\rightarrow \A_0$ for some $C$-full set $\A_0$, then $\A_0 \in \mathcal{K}(C, \omega)$.
\end{itemize}
\end{lemma}

The following lemma shows the continuity of functional $\Phi(Q)$ for $ Q \in \mathscr{K}(C, \omega)$.  
\begin{lemma}\label{conPhi}
Let $\{\A_j\}_{j\in \mathbb{N}} \subset \mathscr{K}(C, \omega)$ be a sequence such that $\A_j\rightarrow \A_0 \in \mathscr{K}(C, \omega)$. Then, $\lim_{j \rightarrow \infty} \Phi(\A_j)=\Phi(\A_0)$.
\end{lemma}
\begin{proof} 
For $q\neq 0$, 
it follows from \cite[Lemma 5.2]{LYZ22} that 
$h_C(\A_j, \cdot)\rightarrow h_C(\A_0, \cdot)$ uniformly on $\omega$, as well as $\rho_C(\A_j, \cdot)\rightarrow \rho_C(\A_0, \cdot)$ uniformly on $\Omega_{C}$. This, in conjunction with formula \eqref{qthdual_Volume}, yields that
\begin{equation*}
\lim_{j \rightarrow \infty} \widetilde{V}_q(\A_j)=\lim_{j \rightarrow \infty} \frac{1}{n} \int_{\Omega_{C}} \rho_C(\A_j, v)^q dv=\frac{1}{n} \int_{\Omega_{C}} \lim_{j \rightarrow \infty} \rho_C(\A_j, v)^q dv=\frac{1}{n} \int_{\Omega_{C}} \rho_C(\A_0, v)^q dv=\widetilde{V}_q(\A_0),
\end{equation*} and $ \lim_{j \rightarrow \infty} \| -h_C(\A_j, \cdot)\|_p =\|-h_C(\A_0, \cdot)\|_p$, due to
\begin{equation*}
\lim_{j \rightarrow \infty} \int_{\omega} (-h_C(\A_j, u))^{p} d\mu(u)=\int_{\omega}\lim_{j \rightarrow \infty}(-h_C(\A_j, u))^{p} d\mu(u)=\int_{\omega} (-h_C(\A_0, u))^{p} d\mu(u).
\end{equation*}
Furthermore, since continuity is preserved for the composition of continuous functions, one can get $\lim_{j \rightarrow \infty} \Phi(\A_j)=\Phi(\A_0)$ for $q\neq 0$.

The case for $q=0$ follows along the same lines.  
\end{proof}
 
We now prove the existence of solutions to the optimization problem for $\Upsilon_Q$ in the case $p\neq 0$. See \cite{LYZ22} for the case $p=0$. 

\begin{lemma}\label{inf}
Let $p,q\in \mathbb{R}$, $p\neq 0,$ and $\mu$ be a nonzero finite Borel measure defined on $\Omega_{C^\circ}$ whose support concentrates on a compact set $\omega \subset \Omega_{C^\circ}$. There exists a $C$-full set $\A_0$ such that $\Phi(\A_0)=\Upsilon_Q$.
\end{lemma}

\begin{proof} 
For $q\neq 0$, as $\Phi$ is homogeneous of degree zero, one can simply let the constraint $Q\in \mathcal{K}(C, \omega)$ in the optimization problem for $\Upsilon_Q$ be the set  $$\mathscr{L}=\left\{ Q \in \mathscr{K}(C, \omega):\,\, \widetilde{V}_q( Q)=1\right\}.$$
	
Let us first verify $\Upsilon_Q>-\infty$. To this end, by using Lemma \ref{Cqb&t2}, one can find a constant $t_0>0$ such that $C \cap H_{t_0} \subset Q$ for each $Q \in \mathscr{L}$. Consequently, $C \cap H^{+}_{t_0} \subset Q$ and  $h_C(Q, \cdot) \geq h_C(C \cap H^{+}_{t_0}, \cdot)$ on $\omega$. As $\widetilde{V}_q( Q)=1$, it follows from \eqref{PhiQ} that 
\begin{equation}\label{Ct0>}
\Phi(Q)=-  \log \|-h_C(Q, \cdot)\|_p
\geq -  \log \|-h_C(C \cap H^{+}_{t_0}, \cdot)\|_p=m_1>-\infty.
\end{equation} As $m_1$ is a constant only depending on $C$ and $\omega$, one gets $\Upsilon_Q>-\infty,$ after taking the infimum of \eqref{Ct0>} over $Q\in \mathcal{L}$.

Let $\{\A_i\}_{i \in \mathbb{N}} \subset \mathscr{L}$ be a limiting sequence such that $\lim_{i \rightarrow \infty}\Phi(\A_i)=\Upsilon_Q$. Applying Lemma \ref{Cqb&t2} to each $\A_i$, one gets $C \cap H_{t_0} \subset \A_i$ for all $i \in \mathbb{N}$, where $t_0$ is a constant depending on $C$ and $\omega$ only. This further implies $\A_i \cap C_{t_0} \subset C_{t_0}$ for each $i \in \mathbb{N}$, where $C_{t_0}=C\cap H^-_{t_0}$ is a convex body. Note that each $\A_i \cap C_{t_0}$ is a compact convex set, and the sequence  $\{ \A_i \cap C_{t_0} \}_{i\in \mathbb{N}}$ is uniformly bounded (by $C_{t_0}$). Applying the Blaschke selection theorem, one can find a subsequence,  denoted by $\{ \A_{i_j} \cap C_{t_0} \}_{j\in \mathbb{N}}$, such that  $\A_{i_j} \cap C_{t_0}$ converges as $j \rightarrow \infty$,  in the Hausdorff metric, to some compact convex set $K \subset \mathbb{R}^n$. Moreover, a closed convex set $\A_0 \subset C$  can be found so that $K=\A_0 \cap H^{-}_{t_0}$ and   $C \cap H_{t_0} \subset \A_0$ (thus, $\A_0$ is $C$-full). Consequently, $\A_{i_j} \cap  H^{-}_{t}\rightarrow \A_0 \cap H^{-}_{t}$ for all $t\geq t_0$, and  from Definition \ref{convergence}, $\A_i \rightarrow \A_0$ as $i \rightarrow \infty$. It follows from Lemma \ref{L56} that $\A_0 \in  \mathscr{K}(C, \omega)$, which further yields, due to the proof of Lemma \ref{conPhi}, $
\widetilde{V}_q(\A_0)=\lim_{i \rightarrow \infty} \widetilde{V}_q(\A_i)
=1.$  In conclusion, one gets $\A_0 \in \mathscr{L}$. Applying Lemma  \ref{conPhi} again, one sees that, for $q\neq 0$,  $ \lim_{j \rightarrow \infty} \Phi(\A_0)=\Phi(\A_{i_j})= \Upsilon_Q.$  

The case for $q=0$ with the constraint set $\mathscr{F}=\left\{ Q \in \mathscr{K}(C, \omega):\,\, \widetilde{\mathbb{E}}( Q)=0\right\}$ follows from similar lines and hence will be omitted.
\end{proof}

\begin{lemma}\label{sol}
Let $p, q\in \mathbb{R}$, $p\neq 0$, and $\mu$ be a nonzero finite Borel measure defined on $\Omega_{C^\circ}$ whose support concentrates on a compact set $\omega \subset \Omega_{C^\circ}$. If there exists a $C$-full set $\A_0\in \mathcal{K}(C, \omega)$ such that $\Phi(\A_0)=\Upsilon_Q$, then 
\begin{equation*}
\mu= \tau_1 \cdot \widetilde{C}_{p, q} ( \A_0, \cdot) \ \ \mathrm{for}\ \ q\neq 0 \ \ \mathrm{and} \ \ \mu=\tau_2 \cdot J^*_p(\A_0, \cdot)  \ \ \mathrm{for}\ \ q=0, 
\end{equation*} where $\tau_1$ and $\tau_2$ are two constants given by \begin{align} \label{tau1-1-1}
\tau_1&=\frac{1}{\widetilde{V}_q(\A_0)}\int_{\omega} (-h_C(\A_0, u))^p d\mu(u) \ \ \mathrm{and}\ \  
\tau_2=\frac{1}{|\Omega_{C}|} \int_{\omega} (-h_C(\A_0, u))^{p} d\mu(u).
\end{align}\end{lemma}

\begin{proof} We only prove the arguments for $q\neq 0$ and omit the proof for the case $q=0$ which follows along the same lines. 

Let $q\neq 0$. We first  verify that $\Upsilon_Q=\Upsilon_f$.  On the one hand,  $\Upsilon_Q  \geq \Upsilon_f$ holds because, for each $Q \in \mathscr{K}(C, \omega)$, $-h_C(Q, \cdot) \in C^{+}(\omega)$ and $Q=[-h_C(Q, \cdot)]$. On the other hand, for all $f \in C^{+}(\omega)$, it follows from (\ref{wul}) that $h_C([f], \cdot) \leq -f$ and hence   $-h_C([f], \cdot) \geq f$ on $\omega$. By (\ref{Phif}) and (\ref{PhiQ}), one gets $\Phi(f) \geq \Phi(-h_C([f], \cdot))=\Phi([f])$. Taking the infimum over $f \in C^+(\omega)$, one has $\Upsilon_f \geq \Upsilon_Q$, and thus $\Upsilon_Q = \Upsilon_f$.

From Lemma \ref{inf}, one gets a $C$-full set $\A_0\in \mathcal{K}(C, \omega)$ such that $\Phi(\A_0)=\Upsilon_Q$. Let $f_0=-h(\A_0, \cdot) \in C^+(\omega)$ and $g \in C(\omega)$. Together with \eqref{pq-dual},  (\ref{ftt}), Theorem \ref{varpq}, and the fact that $\Upsilon_Q=\Upsilon_f$, one has 	
\begin{align*}
0=&\left.\dfrac{d}{dt} \Phi(f_t) \right|_{t=0}\\& 
=\left.\dfrac{d}{dt} \left( - \log \|f_t\|_p+\frac{1}{q}\log \widetilde{V}_q([f_t])  \right)
\right|_{t=0}\\
&=-\frac{\|f_0\|_p^{-p}}{p} \int_{\omega}  g(u)d\mu(u)+\frac{1}{p \widetilde{V}_q(\A_0)}  \int_{\omega} g(u) f_0(u)^{-p} d\widetilde{C}_q(\A_0, u)\\
&=-\frac{\|f_0\|_p^{-p}}{p} \int_{\omega}  g(u)d\mu(u)+\frac{1}{p \widetilde{V}_q(\A_0)}  \int_{\omega} g(u) d\widetilde{C}_{p,q}(\A_0, u). 
\end{align*}
This further yields 
$$ \|f_0\|_p^{-p} \int_{\omega}  g(u)d\mu(u)=\frac{1}{\widetilde{V}_q(\A_0)}  \int_{\omega} g(u) d\widetilde{C}_{p,q}(\A_0, u).$$
Since $g \in C(\omega)$ is arbitrary, one has $\mu=\tau_1 \cdot \widetilde{C}_{p, q} ( \A_0, \cdot)$ as desired.  
\end{proof}

{\noindent \em Proof of Theorem \ref{cd1}.} We only prove this theorem for $q\neq 0$ and omit the proof for $q=0$ which follows along the same lines.

Let $q\neq 0$ and $p \neq q$. Consider
a nonzero finite Borel measure $\mu$ on $\Omega_{C^\circ}$,
with its support concentrated on a compact subset $\omega \subset \Omega_{C^\circ}$. 
Due to Lemma \ref{inf} and  Lemma \ref{sol},  there exists a $C$-full set $\A_0 \in \mathcal{K}(C, \omega)$ satisfying that $\Phi(\A_0)=\Upsilon_Q$ and  $\mu=\tau_1 \cdot \widetilde{C}_{p, q} ( \A_0, \cdot)$ with   $\tau_1$ given by \eqref{tau1-1-1}.  Note that  the $(p, q)$-th dual curvature measure is homogeneous of degree $q-p$ (see  \eqref{homCpq}). This 
implies $\mu=\widetilde{C}_{p, q} ( \A, \cdot)=\widetilde{C}_{p, q} (\tau_1^{\frac{1}{q-p}}\A_0, \cdot)$ for $p \neq q$ by letting $\A=\tau_1^{\frac{1}{q-p}}\A_0$. \qed  
 
\vskip 2mm We now discuss the $L_p$ Alexandrov problem for $C$-compatible sets. Recall that $J^*(\A, \cdot)=\widetilde{C}_0(\A, \cdot)$. 
Note that the underline cone in $\widetilde{C}_0(\cdot, \cdot)$ can be clearly identified through the $C$-compatible sets.
For example, if $\A$ is a $C$-compatible set, its copolar $\copolar\subset C^{\circ}$ is a $C^{\circ}$-compatible set and hence   $\widetilde{C}_0(\copolar, \cdot)$ will be understood as the $0$-dual curvature measure of $\copolar$ with the underline cone $C^{\circ}.$ In  \cite{LYZ22}, such a measure is called the Alexandrov integral curvature measure for $\A\in \Ln$, namely 
$$J(\A, \cdot)=J^*(\A_C^\diamond, \cdot)=\widetilde{C}_0(\copolar, \cdot).$$ Similarly, one can also define the $L_p$ Alexandrov integral curvature measure of $\A$ by  $$J_p(\A, \cdot)=J^*_p(\A_C^\diamond, \cdot) \ \ \ \text{for}\,\,\,p\in \R,$$ and in this formula, the underline cone for $J_p(\cdot, \cdot)$ and $J^*_p(\cdot, \cdot)$ are $C$ and  $C^{\circ}$, respectively. Hence, the measure $J_p(\A, \cdot)$ is defined on (Borel subsets of) $\Omega_{C}$. The following $L_p$ Alexandrov problem for $C$-compatible sets can be posed. See \cite{LYZ22} for the case $p=0$.  

\begin{problem}[The $L_p$ Alexandrov problem for $C$-compatible sets]
For $0\neq p \in \mathbb{R}$ and a nonzero finite Borel measure $\nu$ defined on  $\Omega_{C}$, under what conditions does there exist a $C$-compatible set $\A$  such that $\nu=J_p(\A, \cdot)?$
\end{problem}

We now provide the existence of solutions to the $L_p$ Alexandrov problem for $C$-determined sets. The case for $p=0$ can be found in \cite[Theorem 8.3]{LYZ22}. 

\begin{theorem}\label{Lp_Alex}
Let $0\neq p \in \mathbb{R}$ and let $\nu$ be a nonzero finite Borel measure defined on $\Omega_C$ whose support concentrates on a compact set $\varpi\subset\Omega_C$. Then there exists a $C$-compatible set $\A\subset C$ such that $\A_{C}^{\diamond}$ is a $C^\circ$-full set and  $J_p(\A,\cdot)=\nu$. 
\end{theorem}

\begin{proof} Applying Theorem \ref{cd1} to $q=0$, $\nu$,   $C^{\circ}$ and the compact set $\varpi\subset C$, a $C^{\circ}$-full set $\mathbb{B}\subset C^{\circ}$  can be found such that $\nu=J^*_{p} (\mathbb{B}, \cdot).$ Let $\A=\mathbb{B}^{\diamond}_{C^{\circ}}\subset C$. Then, $\A$ is a $C$-compatible set with $\mathbb{B}=\copolar.$ Thus, for $p\neq 0$, one has, 
$$J_p(\A, \cdot)=J_p^*(\copolar, \cdot)=J_p^*(\mathbb{B}, \cdot) = \nu.$$ This completes the proof. \end{proof}

\subsection{The \texorpdfstring{$L_p$}{} dual Minkowski problem
for \texorpdfstring{$(C,p,q)$}{}-close sets}

In this subsection, the existence of solutions to the $L_p$ dual Minkowski problem for $p\leq 0$ and $p<q$ is provided. Such an existence result relies on the approximation technique used by Schneider in \cite{Sch21}. This requires some preparation for notations, which we will follow those in \cite{LYZ22}. 

Recall that there exists a $\xi \in \Omega_{C}$ (fixed), such that $$C_t=C\cap H^-_{t}=C \cap H^-(\xi, t)$$ is bounded for all  $t>0$. Denote by $B_n$ the unit Euclidean ball in $\Rn.$ Let $b(\A)$ \cite{Sch23} be the distance of $\A \in \Ln$ to the origin $o$, namely,  
\begin{equation*}
b(\A)=\min \{r>0: rB_n \cap \A\neq \emptyset\}.
\end{equation*} It is easily checked that \begin{equation}\label{hb}
-h_C(\A, u)\leq b(\A) \,\,\,\text{for}\,\,\, u \in \Oc\,\,\,\text{and}\,\,\,\rho_C(\A, v)\geq b(\A)\,\,\,\text{for}\,\,\, v\in \Omega_C.
\end{equation} 
For $\tau>0$, let  $\overline{\omega}(\tau)=\{u \in \Oc: \delta_C(u) \geq \tau\},$ where 
\begin{equation*}
\delta_C(u)=\min \{\angle (u, \bar{u}):\ \widetilde{u} \in \partial \Oc\}
\end{equation*} with $\angle (u, \widetilde{u})$  the angle between $u$ and $\widetilde{u}$. Thus, $\delta_C(u)$ is the spherical distance of $u\in \Oc$ to $\partial \Oc$.  

The following lemma provides upper and lower bounds for $b(\A).$

\begin{lemma}\label{blu}
Let $\mu$ be a nonzero finite Borel measure defined on $\Oc$ and let $\A \in \Ln$ be such that $\mu=\widetilde{C}_{p ,q}(\A, \cdot)$. Let
$\tau>0$ be such that $\overline{\omega}(\tau)\subset \Oc$ and $$\widetilde{C}_{p, q}(\A, \overline{\omega}(\tau))=s_0>0.$$ For  $p\leq 0$ and $p<q$, there exist constants $0<\beta_0 <\beta _1<\infty$, depending on $C$ and $\tau$ only,  such that $$\beta_0\leq b(\A)\leq \beta_1.$$ 
\end{lemma}
\begin{proof} Let $p \leq 0$ and $p<q$. We first show the upper bound for $b(\A)$. For the fixed $\xi \in \Omega_C$, one can find $t>0$ small enough so that  
$U_t=\{v \in \Omega_C: \angle(\xi, v)<t\}\subset \Omega_C$. Set $$\overline{U}_{t/2}=\{v \in \Omega_C: \angle(\xi, v)\leq t/2\},$$ which is
a compact set and satisfies $\xi \in \overline{U}_{t/2} \subset U_t \subset \Omega_C$. It follows from \eqref{radial&Gauss}, \eqref{hb} and Proposition \ref{integral}
that 
\begin{align*}
\mu(\Oc)&=\widetilde{C}_{p, q}(\A, \Oc)\\
&=\frac{1}{n} \int_{\Omega_C} [-h_C(\A, {\pmb{\alpha}}_{\A}(v))]^{-p}\rho_C(\A, v)^q dv\\
&=\frac{1}{n} \int_{\Omega_C} [-{\pmb{\alpha}}_{\A}(v)\cdot v]^{-p}\rho_C(\A, v)^{q-p} dv\\
&\geq b(\A)^{q-p} \cdot \frac{1}{n} \int_{\Omega_C} [-{\pmb{\alpha}}_{\A}(v)\cdot v]^{-p} dv\\
&\geq b(\A)^{q-p} \cdot \frac{1}{n} \int_{\overline{U}_{t/2}} [-{\pmb{\alpha}}_{\A}(v)\cdot v]^{-p} dv.
\end{align*} 

Note that $\pmb{\alpha}_{\A}(v) \in \overline{\Oc}$, the closure of $\Oc$. The compactness of $\overline{U}_{t/2}$ and $\overline{\Oc}$, together with the continuity for inner product, yields that  $-{\pmb{\alpha}}_{\A}(v)\cdot v\in[a_1,1]$, where $a_1$ is a positive constant depending on $\xi$ and $t$ only.  
Therefore, we have
\begin{align*}
\mu(\Oc) 
\geq b(\A)^{q-p}\cdot\frac{a_1^{-p}}{n} \int_{\overline{U}_{t/2}}dv.
\end{align*} 
This further implies, as $p<q,$
\begin{align*}
 b(\A)
\leq \left(\frac{a_1^{-p}}{n\mu(\Oc)} \int_{\overline{U}_{t/2}}dv\right)^{\frac{1}{p-q}}:=\beta_1,  
\end{align*} where $\beta_1$ is a constant depending on $C$, $t$, and $\xi.$ Note that both $\xi$ and $t$ can be uniquely determined by $C$ only, so $\beta_1$ is a constant depending on $C$ only. 

Now let us prove the lower bound for $b(\A).$ Since $b(b(\A)^{-1}\A)=1$ for $\A \in \Ln$, 
without loss of generality, we consider $b(\A)=1$. Thus, there is a point $y \in \partial \A \cap B_n$. For any $x \in \nu_\A^{-1}(\overline{\omega}(\tau))$, we have $\nu_{\A}(x)\cap \overline{\omega}(\tau) \neq \emptyset$, and there exists $u \in \overline{\omega}(\tau)$ such that $h_C(\A, u)=x\cdot u$. 
Thus, the supporting hyperplane $H(u, h_C(\A, u))$ of $\A$ at $x$ satisfies $\A \subset H^{-}(u, h_C(\A, u))$, and it separates the origin $o$ and $y$ (as $y \in \A$); see
Figure \ref{fig1}. 
\begin{figure}[ht]
\centering    \includegraphics[width=0.62\textwidth]{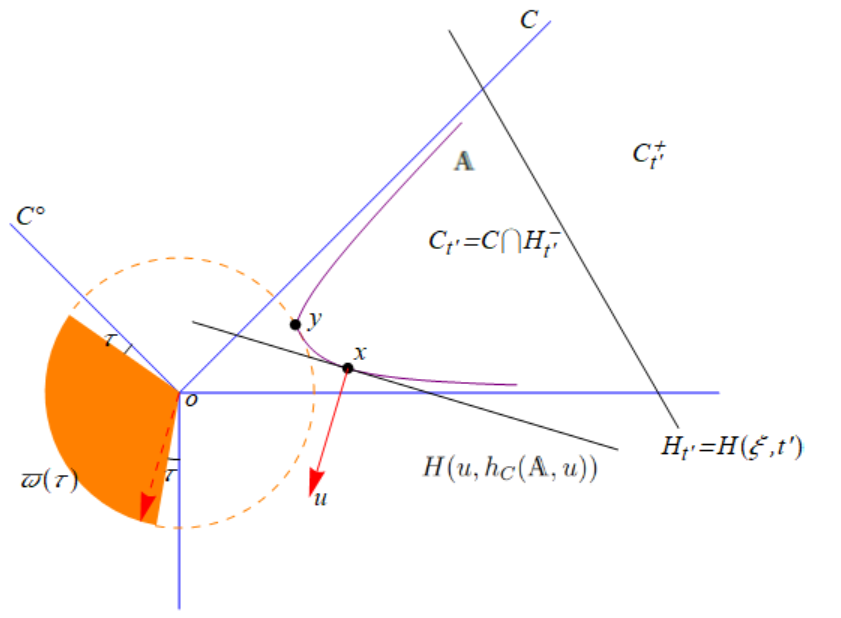}
\caption{}
\label{fig1}
\end{figure}

For any $u \in \overline{\omega}(\tau)$, one has
$ h_C(\A, u) \geq y\cdot u\geq -1$. 
Define the $C$-determined set (by the compact set $\overline{\omega}(\tau)$) 
\begin{equation*}
\mathbb{B}=C\cap \bigcap_{u\in \overline{\omega}(\tau)}  \Big\{z: z\cdot u\leq -1\Big\}.
\end{equation*}  Thus, by \cite[Lemma 8]{Sch18}, it follows that there exists $t'>0$ depending only on $C$ and $\tau$ such that $\nu_{\mathbb{B}}^{-1}(\overline{\omega}(\tau)) \subset C_{t'}$. Furthermore, one has  $\nu_{\A}^{-1}(\overline{\omega}(\tau)) \subset C_{t'}$. Thus,  for $v \in {\pmb{\alpha}}^*_{\A}(\overline{\omega}(\tau))$ (and hence $x=\rho_C(\A, v)v\in \nu_{\A}^{-1}(\overline{\omega}(\tau))$), one gets
$\rho_C(\A, v)\leq \rho_C(C^+_{t'}, v),
$ where $C^+_{t'}=C \cap H^+(\xi, {t'})$ is a $C$-compatible set. Note that $p\leq 0$ yields $[-{\pmb{\alpha}}_{\A}(v)\cdot v]^{-p}\leq 1.$ Together with \eqref{radial&Gauss} and Proposition \ref{integral},   one can get, for $q>p$,  \begin{align}
\widetilde C_{p, q}(\A, \overline{\omega}(\tau))&= \frac{1}{n} \int_{{\pmb{\alpha}}^*_{\A}(\overline{\omega}(\tau))} [-h_C(\A, {\pmb{\alpha}}_{\A}(v))]^{-p}\rho_C(\A, v)^q dv\nonumber\\
&=\frac{1}{n} \int_{{\pmb{\alpha}}^*_{\A}(\overline{\omega}(\tau))} [-{\pmb{\alpha}}_{\A}(v)\cdot v]^{-p}\rho_C(\A, v)^{q-p} dv\nonumber \\
&\leq \frac{1}{n} \int_{{\pmb{\alpha}}^*_{\A}(\overline{\omega}(\tau))} \rho_C(\A, v)^{q-p} dv\nonumber \\
&\leq \frac{1}{n} \int_{{\pmb{\alpha}}^*_{\A}(\overline{\omega}(\tau))} \rho_C(C^+_{t'}, v)^{q-p} dv\nonumber \\
&\leq \frac{1}{n} \int_{\Omega_C} \rho_C(C^+_{t'}, v)^{q-p} dv=c_2, \label{Form-1-bk=1}
\end{align}  where $c_2$ is a constant depending on $\tau$ and $C$ only. 

The inequality \eqref{Form-1-bk=1} is proved under the assumption that $b(\A)=1.$ For general $\A \in \Ln$, one can consider $b(\A)^{-1}\A$ which gives $b(b(\A)^{-1}\A)=1$. Then,  \eqref{Form-1-bk=1} is applied to $b(\A)^{-1}\A$ to get 
\begin{equation*}
c_2 \geq \widetilde C_{p, q}(b(\A)^{-1}\A, \overline{\omega}(\tau)) =b(\A)^{p-q}\widetilde C_{p, q}(\A, \overline{\omega}(\tau))=b(\A)^{p-q}\cdot s_0,
\end{equation*} where we have used \eqref{homCpq}. After a rearrangement, one gets, for $p\leq 0$ and $p<q$,  $$b(\A)\geq \Big(\frac{s_0}{c_2}\Big)^\frac{1}{q-p}=\beta_0.$$ This concludes the proof.
\end{proof}

Let $\mu$ be a nonzero finite Borel measure defined on $\Omega_{C^\circ}$. Denote by $\mathcal{O}$ the family of open sets of $\Omega_{C^\circ}$ satisfying that if $\omega \in \mathcal{O}$ then $\emptyset \neq \omega \subset \overline{\omega} \subset \Omega_{C^\circ}$ and $\mu(\omega)\neq 0$, where $\overline{\omega}$ refers to the closure of $\omega$ and hence $\overline{\omega}$ is compact. For $\omega \in \mathcal{O}$, let $\mu_{\omega}$ be the Borel measure obtained from $\mu$ restricted on $\omega$, i.e., $\mu_{\omega}(\cdot)=\mu(\omega \cap \cdot)$. Clearly, $\mu_{\omega}$ is finite and its support concentrates on the compact set $\overline{\omega}.$ Applying Theorem \ref{cd1} to measure $\mu_{\omega}$ yields the existence of a $C$-full set $\A_{\omega}\in \mathcal{K}(C, \overline{\omega})$ such that $$\mu_\omega=\widetilde{C}_{p, q} (\A_\omega, \cdot)$$ for $ p, q \in \mathbb{R}\setminus \{0\}$ and $p\neq q$.  

Let $\tau>0$ be a constant such that $\mu(\overline{\omega}(\tau))>0$. 
Taking a sequence $\{\omega_j\}_{j \in \mathbb{N}} \subset \mathcal{O}$ of open sets in $\Omega_{C^\circ}$  such that $\overline{\omega}(\tau) \subset \omega_1$, 
$$\omega_j \subset \overline{\omega_j} \subset \omega_{j+1}\subset \Omega_{C^\circ} \ \ \mathrm{for\  all\ } \ j \in \mathbb{N},$$   and $ \cup_{j \in \mathbb{N}} {\omega_j}=\Omega_{C^\circ}.$  
Let $\mu_j=\mu(\omega_j\cap \cdot)$,  
and $\mu_j(\eta)\rightarrow \mu(\eta)$ for each Borel set $\eta\subset\Oc$. 
For $p\leq 0$ and $p<q$,
since each $\overline{\omega_j}$ is compact and each $\mu_j$ is a nonzero finite Borel measure whose support is concentrated on $\overline{\omega_j}$, it follows from Theorem \ref{cd1} that, 
for any $j \in\mathbb{N}$, 
there exists a $C$-full set $\A_j\in \mathcal{K}(C, \overline{\omega_j})$ such that
$$\,d\mu_j=\,d\widetilde{C}_{p, q}(\A_j, \cdot)=(-h_C(\A_j, \cdot))^{-p}\,d\widetilde{C}_q(\A_j, \cdot).$$ Furthermore, 
\begin{equation*}
s_0=\mu(\overline{\omega}(\tau))=\mu(\overline{\omega}(\tau) \cap \omega_j)=\mu_j(\overline{\omega}(\tau))=\widetilde C_{p, q}(\A_j, \overline{\omega}(\tau)). \end{equation*} From Lemma \ref{blu} and the fact that $\mu$ is a finite measure on $\Oc$, one has, for any $j\in \mathbb{N}$,  
\begin{equation*}
\overline{\beta_0}=\left(\frac{s_0}{c_2} \right)^\frac{1}{q-p} \leq b(\A_j)\leq \left( \frac{a_1^{-p}}{n\mu_j(\Oc)}\right)^\frac{1}{p-q} \leq \left( \frac{a_1^{-p}}{n\mu(\Oc)}\right)^\frac{1}{p-q}=\overline{\beta_1},
\end{equation*}
with $a_1$ and $c_2$ the constants depending only on $C$ and $\tau$ given in the proof of Lemma \ref{blu}. 
Thus, for $p\leq 0$ and $p<q$, the sequence $\{b(\A_j)\}_{j\in \mathbb{N}}$ is uniformly bounded from below by a positive constant and also uniformly bounded from above.   In particular, we have the following result.
\begin{corollary}\label{bjb}
Let $p,q\in\mathbb{R}$ be such that $p\leq 0$ and $p<q$. For all $t>\overline{\beta_1}$ and $j \in \mathbb{N}$, 
\begin{equation*}
\A_j \cap C_t \neq \emptyset.
\end{equation*}
\end{corollary}

We now prove Theorem \ref{Cpq-exis} by using the approximation method  in \cite{LYZ22, Sch21}. We will keep our proof brief, and more details for the approximation method can be found in \cite{LYZ22, Sch21}. 

\vskip 2mm 
{\noindent \em Proof of Theorem \ref{Cpq-exis}.} 
Let $p,q\in\mathbb{R}$ be such that $p\leq 0$ and $p<q$ and let $\{\omega_j\}_{j \in \mathbb{N}}$ be given as above. Let $\overline{\beta_1}<t_1<t_2<\cdots< t_k< \cdots $ be an increasing sequence such that $\lim_{k\rightarrow \infty} t_k =+\infty.$ From Corollary \ref{bjb}, one has $$\emptyset \neq \A_j \cap C_{t_k} \subset C_{t_k}$$ for each $j \in \mathbb{N}$ and $k \in \mathbb{N}$. Now we can use \cite[Lemma 7.2]{LYZ22}, which was originally given in \cite[Theorem 5]{Sch18} (see also the proof of \cite[Theorem 1]{Sch21}). That is, given $k \in \mathbb{N}$, without loss of generality, we can assume that 
\begin{equation*}
\A_j \cap C_{t_k} \rightarrow M_k\,\,\, \text{as}\,\,\, j \rightarrow \infty,
\end{equation*} 
in the Hausdorff metric, where  $M_k$ is a compact convex set, and hence $h(\A_j \cap C_{t_k}, \cdot) \rightarrow h(M_K , \cdot)$ uniformly on $S^{n-1}$,
where $h(K, u)=\max\{ x\cdot u: x \in K\}$ denotes the support function of a compact convex set $K$.
Note that $M_l=M_k\cap C_{t_l}$ for $1\leq l < k$. Now let $\A=\cup_{k \in \mathbb{N}} M_k$. It follows from Lemma \ref{blu} that $o\notin\A$ and $\A$ is a closed convex set such that $\A\cap C_{t_k}=M_k$ for $k\in \mathbb{N}$.  
  
Next we show that, for $p\leq 0$ and $p<q$,  $\mu=\widetilde{C}_{p, q}(\A, \cdot)$ on $\Omega_{C^\circ}$.  
For each $\omega \in \mathcal{O}$, let $j_0 \in \mathbb{N}$ be the smallest number such that $\omega \subset \omega_{j_0}$. Thus, there exists $k_0 \in \mathbb{N}$ such that $\nu^{-1}_{\A}(\omega)=\nu^{-1}_{M_k}(\omega)$ for all $k \geq k_0$ (see e.g., \cite[Lemma 7.2 (iv)]{LYZ22}). On the compact set $\overline{\omega_{j_0}}$, $h(\A_j \cap C_{t_k}, \cdot)$ and $h(M_K , \cdot)$ are continuous negative functions such that $h(\A_j \cap C_{t_k}, \cdot) \rightarrow h(M_K , \cdot)$ uniformly.  Lemma \ref{L56} yields that $\widehat{\A}_j \rightarrow \widehat{\A}$ as $j \rightarrow \infty$, where 
\begin{equation}\label{AjA}
\widehat{\A}_j=[C, \overline{\omega_{j_0}}, h(\A_j \cap C_{t_k}, \cdot)] \in \mathcal{K}(C, \overline{\omega_{j_0}})\,\,\,\text{and}\,\,\,
\widehat{\A}=[C, \overline{\omega_{j_0}}, h(M_k, \cdot)] \in \mathcal{K}(C, \overline{\omega_{j_0}}).
\end{equation} Therefore, 
$h_C(\widehat{\A}_j, \cdot)\rightarrow h_C(\widehat{\A}, \cdot)$ uniformly on $\overline{\omega_{j_0}}$. 

For $q\neq 0$, it has been proved in \cite[Lemma 5.3]{LYZ22} that $\widetilde{C}_q(\widehat{\A}_j, \cdot) \rightarrow \widetilde{C}_q(\widehat{\A}, \cdot)$ weakly  as $j \rightarrow \infty$. Such a weak convergence also holds for $q=0$, namely $J^*(\widehat{\A}_j, \cdot) \to J^*(\widehat{\A}, \cdot)$ weakly as $j \to \infty$, where $J^*(\A, \cdot)=\widetilde C_0(\A, \cdot)$. Indeed, for any continuous and bounded function $f: \Oc\to \mathbb{R}$, by \cite[Lemma 4.2]{LYZ22}, the dominated convergence theorem and \eqref{Borel&q0}, one can obtain that
\begin{equation*}
\lim_{j \to \infty} \int_{\Oc} f(u)dJ^*(\widehat{\A}_j, u)=\lim_{j \to \infty} \int_{\Omega_{C}} f(\alpha_{\widehat{\A}_j}(v))dv=\int_{\Omega_{C}} f(\alpha_{\widehat{\A}}(v))dv=\int_{\Oc} f(u)dJ^*(\widehat{\A}, u).
\end{equation*} 

Combining with \cite[Theorem 4.5.1]{Ash72}, it follows that for an open set $\omega \subset \omega_{j_0}$,  
\begin{align}\label{wco}
\widetilde{C}_{p, q}(\A, \omega)\notag
&= \int_{\omega} (-h_C(\A , u))^{-p} d\widetilde{C}_q(\A, u) \\
&=\int_{\omega} (-h_C(\widehat{\A} , u))^{-p} d\widetilde{C}_q(\widehat{\A}, u) \notag \\
&\leq  \liminf_{j \rightarrow \infty} \int_{\omega} (-h_C(\widehat{\A}_j , u))^{-p} d\widetilde{C}_q(\widehat{\A}_j, u)  \notag \\
&=\liminf_{j \rightarrow \infty} \int_{\omega} (-h_C(\A_j, u))^{-p} d\widetilde{C}_q(\A_j , u) \notag  \\
&=\liminf_{j \rightarrow \infty} \widetilde{C}_{p, q}(\A_j, \omega)
=\liminf_{j \rightarrow \infty} \mu_j(\omega)=\mu({\omega}).
\end{align} 

Let $\beta \subset \omega_{j_0}$ be a closed set, and   $\{\beta_l\}_{l \in \mathbb{N}}$ be a sequence of open neighbourhoods of $\beta$ such that $\beta\subset \beta_{l_1} \subset \beta_{l_2} \subset \omega_{j_0}$ for any $l_1>l_2$ and $\beta=\cap_{l\geq 1} \beta_l.$ 
It follows from (\ref{wco}) that    $\widetilde{C}_{p ,q}(\A, \beta_l) \leq \mu(\beta_l)$ and then 
$\widetilde{C}_{p ,q}(\A, \beta) \leq \mu(\beta)$. Similarly,   \cite[Theorem 4.5.1]{Ash72} also implies that  
\begin{align*}
\widetilde{C}_{p, q}(\A, \beta)\notag
&= \int_{\beta} (-h_C(\A , u))^{-p} d\widetilde{C}_q(\A, u) 
  \notag \\
&=\int_{\beta} (-h_C(\widehat{\A} , u))^{-p} d\widetilde{C}_q(\widehat{\A}, u)\notag \\
&\geq  \limsup_{j \rightarrow \infty} \int_{\beta} (-h_C(\widehat{\A}_j , u))^{-p} d\widetilde{C}_q(\widehat{\A}_j, u) \notag \\
&=\limsup_{j \rightarrow \infty} \int_{\beta} (-h_C(\A_j, u))^{-p} d\widetilde{C}_q(\A_j , u) \notag  \\
&=\limsup_{j \rightarrow \infty} \widetilde{C}_{p, q}(\A_j, \beta)=\limsup_{j \rightarrow \infty}\mu_j(\beta)=\mu({\beta}).
\end{align*} 
This, together with \eqref{wco}, yields that 
$\widetilde{C}_{p, q}(\A, \beta)=\mu({\beta})$ for any closed set $\beta \subset \omega_{j_0}$. This indeed further implies 
$\widetilde{C}_{p ,q}(\A, \beta)=\mu({\beta})$ for any closed set $\beta \subset \Omega_{C^\circ}$. 
 Consequently, for $p\leq 0$ and $p<q$,   
$\widetilde{C}_{p, q}(\A, \cdot)=\mu$ on $\Omega_{C^\circ}$. 

Finally, for $p\leq 0$ and $p<q$, 
we can check that  $\A$ is $(C, p, q)$-close. This is an easy consequence of  \eqref{Cpqclose}, Definition \ref{pqdual} and the fact that
$\mu$ is a finite Borel measure, namely, 
\begin{equation*}
\widetilde{C}_{p, q}(\A, \Omega_{C^\circ})=\mu(\Omega_{C^\circ})<\infty.
\end{equation*}
The proof is completed.\qed

When $q=n$, the $L_p$ dual Minkowski problem reduces to the $L_p$ Minkowski problem  \cite{Sch18, YYZ22} up to a multiplication of constant $\frac{1}{n}$. This problem can be described as follows: 

\vskip 2mm \noindent \textbf{The $L_p$ Minkowski problem for $C$-compatible sets:} {\em Given a nonzero finite Borel measure $\mu$ defined on $\Omega_{C^\circ}$ and a real number $p$, under what conditions does there exist a $C$-compatible set $\A\in \Ln$ such that $\mu=S_{n-1, p}(\A, \cdot)$?} Here $S_{n-1, p}(\A, \cdot)$ is the $L_p$ surface area measure of $\A$ which can be defined similar to the one in \eqref{Lp-measure-support}. 

With the help of Theorem \ref{Cpq-exis}, the subsequent corollary can be obtained. The case for  $p=0, 1$ can be seen in \cite{Sch18, Sch21}, and for $p \in (0, 1)$ is considered in \cite{AYZ}.
 
\begin{corollary}
Let $p< 0$ and let $\mu$ be a nonzero finite Borel measure defined on $\Oc$. There exists a $(C, p, n)$-close set such that $\mu=S_{n-1, p}(\A, \cdot)$.
\end{corollary}
 
When $q=0$, $\widetilde C_{p, q}(\A \cdot)=J^*_p(\A, \cdot)$ for $\A \in \Ln$ and the $L_p$ Alexandrov integral curvature measure of $\A$ is given by $$J_p(\A, \cdot)=J^*_p(\A^\diamond_C, \cdot).$$ A direct consequence of Theorem \ref{Cpq-exis} for  $q=0$ and $p<0$ yields the following existence result to the $L_p$ Alexandrov problem for $p<0$ in a more general setting. Its proof follows along the same pattern as that of Theorem \ref{Lp_Alex}.
\begin{theorem}
Let $0>p \in \mathbb{R}$, and $\nu$ be a nonzero finite Borel measure defined on $\Omega_C$. Then there exists a $C$-compatible set $\A\subset C$ such that $\A_{C}^{\diamond}$ is a $(C^\circ, p, 0)$-close set and  $J_p(\A,\cdot)=\nu$. 
\end{theorem}
\begin{proof}
For $q=0$, applying Theorem \ref{Cpq-exis} to measure $\nu$ and cone $C^\circ$, one can find a $(C^\circ, p, 0)$-close set $\B\subset C^{\circ}$ such that $\nu=J^*_p(\B, \cdot)$ for $p<0$. Let $\A=\B^\diamond_{C^\circ} \subset C$, and hence $\A$ is $C$-compatible and $\B=\A_C^\diamond$. Therefore, for $p<0$, one has
$$J_p(\A, \cdot)=J_p^*(\copolar, \cdot)=J_p^*(\mathbb{B}, \cdot) = \nu,$$ as desired. This completes the proof.
\end{proof}
 
\section{Uniqueness of solutions to the \texorpdfstring{$L_p$}{} dual Minkowski problem for \texorpdfstring{$C$}{}-compatible sets} \label{sec5}

The main objective of this section is to provide some uniqueness results regarding the solutions to the $L_p$ dual Minkowski problem  for $C$-compatible sets. 

Recall that, when $p=0$, the $L_p$ dual Minkowski problem becomes the dual Minkowski problem raised by Li, Ye and Zhu in \cite{LYZ22}, where they established the existence of solutions to the dual Minkowski problem for the case of $(C, q)$-close sets when $q>0$. However, the uniqueness was not discussed, and it will be our first problem of interest in this section. To fulfill this goal, we shall need the following lemma motivated by Zhao \cite{Zhao17}.

\begin{lemma}\label{Gam123}
Let $\B_1$ and $\B_2$ be two $(C, q)$-close sets for $q>0$. Suppose that  
\begin{align*}
&\eta_1=\left\{  u \in \Omega_{C^\circ} : h_C(\B_1, u)>h_C(\B_2, u)\right\} \neq \emptyset, \\ 
& \eta_2=\left\{  u \in \Omega_{C^\circ} : h_C(\B_1, u)<h_C(\B_2, u)\right\}\neq \emptyset, \\
&\eta_3=\left\{  u \in \Omega_{C^\circ} : h_C(\B_1, u)=h_C(\B_2, u)\right\}\neq \emptyset.
\end{align*} 
Then, the following hold:
\begin{itemize}
\item[(a)] $\rho_C({\B}_1, v)> \rho_C({\B_2}, v)$ if $v \in {{\pmb{\alpha}}^{\ast}_{\B_2} (\eta_2)}$ and $\rho_C(\B_1, v) \leq \rho_C({\B_2}, v)$
if $v \in {{\pmb{\alpha}}^{\ast}_{\B_1} (\eta_1 \cup \eta_3)}$;
\item[(b)] ${{\pmb{\alpha}}^{\ast}_{\B_2} (\eta_2)} \subset {{\pmb{\alpha}}^{\ast}_{\B_1} (\eta_2)}$;
\item[(c)]  $\mathcal{H}^{n-1} ({{\pmb{\alpha}}^{\ast}_{\B_1} (\eta_2)} )>0$.
\end{itemize}
\end{lemma}
\begin{proof}
We prove the above arguments by contradiction. 

\vskip 2mm \noindent For (a), assume that there is some $v_0 \in{{\pmb{\alpha}}^{\ast}_{\B_2}(\eta_2)}$ such that $\rho_C({\B_1}, v_0) \leq \rho_C({\B_2}, v_0)$. Then, one can find a vector $u_0 \in \eta_2\subset \Oc$ such that 
$\rho_C({\B_2}, v_0)(u_0 \cdot v_0)=h_C(\B_2, u_0)$. As $\B_2$ is a $(C, q)$-close set, one has $o\notin   \B_2$, and  $h_C(\B_2, u_0)<0$ for $u_0\in \Oc.$ In particular, $u_0 \cdot v_0< 0$. This further gives 
$$h_C(\B_2, u_0)=\rho_C({\B_2}, v_0)(u_0 \cdot v_0)\leq  \rho_C({\B_1}, v_0)(u_0 \cdot v_0) \leq h_C(\B_1, u_0),$$
a contradiction with the definition of $\eta_2$. Therefore, $\rho_C({\B_1}, v)> \rho_C({\B_2}, v)$
if $v \in {{\pmb{\alpha}}^{\ast}_{\B_2} (\eta_2)}$. 

The same procedure can be used to obtain the second part of (a). In fact, if $\rho_C(\B_1, v_0') > \rho_C({\B_2}, v_0')$
for some $v_0' \in {{\pmb{\alpha}}^{\ast}_{\B_1} (\eta_1 \cup \eta_3)}$. Then a vector $u_0' \in \eta_1\cup \eta_3$ can be found so that $\rho_C(\B_1, v_0')(u_0'\cdot v_0')=h_C(\B_1, u_0')$ and again $u_0'\cdot v_0'<0$ due to $u_0' \in \Oc$. 
Therefore,
\begin{equation*}
h_C(\B_1, u_0')=\rho_C(\B_1, v_0')(u_0'\cdot v_0')<\rho_C(\B_2, v_0')(u_0'\cdot v_0')\leq h_C(\B_2, u_0'),
\end{equation*}
which contradicts with  $u_0'\in \eta_1\cup \eta_3$. Thus,  
$\rho_C(\B_1, v) \leq \rho_C({\B_2}, v)$
if $v \in {{\pmb{\alpha}}^{\ast}_{\B_1} (\eta_1 \cup \eta_3)}$.
 
\vskip 2mm \noindent 
For (b), assume that ${{\pmb{\alpha}}^{\ast}_{\B_2} (\eta_2)} \not\subset {{\pmb{\alpha}}^{\ast}_{\B_1} (\eta_2)}$, i.e., there is $v_0 \in {\alpha^{\ast}_{\B_2} (\eta_2)}$ but $v_0 \not\in  {{\pmb{\alpha}}^{\ast}_{\B_1} (\eta_2)}$. This implies $v_0 \in {{\pmb{\alpha}}^{\ast}_{\B_2} (\eta_2)} \cap {{\pmb{\alpha}}^{\ast}_{\B_1} (\eta_1 \cup\eta_3 )}$, which is not possible due to (a). 

\vskip 2mm \noindent For (c), assume that  
$\mathcal{H}^{n-1} ({{\pmb{\alpha}}^{\ast}_{\B_1} (\eta_2)} )=0$.
Then ${{\pmb{\alpha}}^{\ast}_{\B_1} (\eta_1 \cup \eta_3)}$ differs from $\Omega_{C}$ by a null set.
Together with (a), one has $\rho_C({\B_1}, v) \leq \rho_C({\B_2}, v)$ for $\mathcal{H}^{n-1}$-almost all $v\in \Omega_{C}$. As the radial functions $\rho_C({\B_1}, \cdot)$ and $\rho_C({\B_2}, \cdot)$ are both continuous on $\Omega _C$, one sees that  $\B_2 \subseteq \B_1$. This further yields  $h_C(\B_1, \cdot) \geq h_C(\B_2, \cdot)$, and thus $\eta_2=\emptyset$. This is a contradiction to the hypothesis, and therefore $\mathcal{H}^{n-1} ({{\pmb{\alpha}}^{\ast}_{\B_1} (\eta_2)})>0.$
\end{proof}
 
With the help of Lemma \ref{Gam123}, the following uniqueness result of the solutions to the dual Minkowski problem for $(C, q)$-close sets can be obtained.
\begin{theorem}\label{unicq}
Let $q>0$, and $\A_1$, $\A_2$ be two $(C, q)$-close sets.  If $\widetilde{C}_q (\A_1, \cdot)=\widetilde{C}_q (\A_2, \cdot)$, then $\A_1=\A_2$.
\end{theorem}
\begin{proof} Let $q>0$. As the measure  $\widetilde{C}_q(\cdot, \cdot)$ is homogeneous of degree $q$, it is enough to prove the argument by assuming that $\A_1$ and $\A_2$ are not dilates of each other. Consequently, there exists $\lambda>0$ such that $\B_1=\lambda \A_1$ and $\B_2=\A_2$ satisfy the assumption in Lemma \ref{Gam123}. It follows from \eqref{qdual} and part (c) of Lemma \ref{Gam123} that $\widetilde{C}_q(\B_1, \eta_2)>0$. As $\widetilde{C}_q(\B_1, \cdot)=\lambda^q\widetilde{C}_q(\A_1, \cdot)$ and $\widetilde{C}_q (\A_1, \cdot)=\widetilde{C}_q (\A_2, \cdot)$, one also has $$\widetilde{C}_q({\A}_1, \eta_2)=\widetilde{C}_q({\A}_2, \eta_2)>0.$$ 
Together with \eqref{qdual} and Lemma \ref{Gam123}, one gets, for $q>0$, 
\begin{align*}
\widetilde{C}_q({\A}_1, \eta_2)
&=\widetilde{C}_q({\A}_2, \eta_2)\\ &=\frac{1}{n} \int_{{\pmb{\alpha}}^{\ast}_{\A_2} (\eta_2)}   \rho_C(\A_2, v) ^q dv
\\ &<\frac{1}{n} \int_{{\pmb{\alpha}}^{\ast}_{{\A}_2} (\eta_2)}  \rho_C(\B_1, v)^q dv\\
& \leq  \frac{1}{n} \int_{{\pmb{\alpha}}^{\ast}_{\B_1} (\eta_2)}  \rho_C(\B_1, v)^q dv\\ &=\widetilde{C}_q(\B_1, \eta_2)=\lambda^q \widetilde{C}_q({\A}_1, \eta_2).
\end{align*} As $\widetilde{C}_q({\A}_1, \eta_2)>0,$ one must have  $\lambda^q >1$ and hence $\lambda >1$. Following the same lines as above, if we let $\B_2=\lambda \A_1$ and $\B_1=\A_2$, one should get $\lambda <1$, which is impossible. This shows that $\A_1=\A_2$ as desired. 
\end{proof}

We now consider the uniqueness of solutions to the $L_p$ dual Minkowski problem when the given measure is  concentrated on finite many directions. This is related to $C$-polytopes, which are $C$-determined sets by finite many unit vectors in $\Oc.$ That is,  $\mathbb{P}\subset C$ is said to be a $C$-polytope if there exist  $u_1, u_2,..., u_m\in \Oc$, such that, $\mathbb{P}\in \mathcal{K}(C, \{u_1, u_2,..., u_m\}).$ Moreover, by \eqref{Cdeter},  
\begin{equation*}
\mathbb{P}=C \cap \bigcap_{i=1}^mH ^{-} \left( u_i, h_C(\mathbb{P}, u_i)\right).
\end{equation*} 
The set $F_i=\mathbb{P} \cap H  \left( u_i, h_C(\mathbb{P}, u_i)\right)$ defines the facet of $\mathbb{P}$ with unit outer normal $u_i\in \Oc$. Let 
\begin{align}\label{ci}
\Delta_{\mathbb{P},i}&=\{tx\in C: t>0 \ \ \mathrm{and}\ \ x\in F_i\}  \ \ \mathrm{and} \ \ 
c_i=\frac{1}{n}\left[-h_C(\mathbb{P}, u_i)\right]^{-p} \int_{\Omega_{C} \cap \Delta_{\mathbb{P},i}}\rho_C(\mathbb{P}, v)^qdv. 
\end{align} 
It can be checked that 
\begin{equation}\label{dis&cpq}
\widetilde{C}_{p, q}(\mathbb{P}, \cdot)=\sum_{i=1}^{m} c_i \delta_{u_i},
\end{equation}
where $\delta_{u_i}$ represents the Dirac measure concentrated at $u_i$.  Indeed, \eqref{dis&cpq} is an easy consequence of the combination with Proposition \ref{integral}, \begin{align*} 
\pmb{\alpha}^*_{\mathbb{P}}(u_i)&=\Omega_{C} \cap \Delta_{\mathbb{P},i}  \ \ \mathrm{up\ to \ a \ set \ of \ spherical \ Lebesgue \ measure}\ 0, \\ \pmb{\alpha}_{\mathbb{P}}(v)&=u_i\ \ \mathrm{for\ almost\ all}\ v \in \Omega_{C} \cap \Delta_{\mathbb{P},i}, 
\end{align*} (following from \eqref{Cdeter} and \eqref{reverse&radial&Gauss}), and the fact that  the spherical Lebesgue measure of $\pmb{\alpha}^*_{\mathbb{P}}(\eta)$ is $0$ if  $\eta \subseteq \Omega_{C^\circ}$ such that $\eta \cap \{u_1, \cdots, u_m\}=\emptyset$. 

Following the technique used in \cite[Theorem 8.3]{LYZ18}, one can get the following uniqueness of solutions to the $L_p$ dual Minkowski problem for discrete measures.
 \begin{theorem}\label{uni&dis} 
Let $\{u_1, u_2,..., u_m\}\subset \Oc$, and $\mathbb{P}_1, \mathbb{P}_2 \in \mathscr{K}(C, \{u_1, u_2,..., u_m\})$  be two $C$-polytopes.  Suppose that $\widetilde{C}_{p, q}(\mathbb{P}_1, \cdot)=\widetilde{C}_{p, q}(\mathbb{P}_2, \cdot)$ on $\Omega_{C^\circ}$. Then $\mathbb{P}_1$ and $\mathbb{P}_2$ are dilates of each other if $p=q$, while $\mathbb{P}_1=\mathbb{P}_2$ if $p<q$.
\end{theorem}

\begin{proof} For $k=1, 2$, by \eqref{ci} and \eqref{dis&cpq},  
\begin{equation*} 
\widetilde{C}_{p, q}(\P_k, \cdot)=\sum_{i=1}^{m} c_{i, k} \delta_{u_i} 
\quad \text{with}\quad
c_{i, k}=\frac{1}{n}\left[-h_C(\P_k, u_i)\right]^{-p} \int_{\Omega_{C} \cap \Delta_{\P_k, i}}  \rho_C(\P_k, v)^qdv.
\end{equation*}

Firstly we let  $p=q$. We need to prove that $\P_1$ and $\P_2$ are dilates of each other. To this end, assume the opposite, i.e., $\P_1$ and $\P_2$ are not dilates of each other, and then $\P_1\neq a \P_2$ for any $a>0.$ Thus, a constant $a_0>0$ can be found such that $\P_1\subsetneq  a_0\P_2$ satisfy the following property: the set  $$\Sigma=\{j\in \{1, \cdots, m\}: h_C(\P_1, u_j)=h_C(a_0\P_2, u_j)\}$$ is a nonempty proper subset of $\{1, \cdots, m\}.$  Clearly, as $\P_1\subsetneq  a_0\P_2$,  $h_C(\P_1, u_j)<h_C(a_0\P_2, u_j)$ for $j\in \{1, \cdots, m\}\setminus \Sigma$. Due to the fact that $\Sigma\subsetneq \{1, \cdots, m\}$, 
there exists $j_0\in \Sigma$  such that  $\Delta_{\P_1, j_0}\subsetneq \Delta_{a_0\P_2, j_0}$. 
Hence, for $p=q$, 
\begin{align}
c_{j_0, 2}&=\frac{1}{n}\left[-h_C(\P_2, u_{j_0})\right]^{-p} \int_{\Omega_{C} \cap \Delta_{\P_2, j_0}}\rho_C(\P_2, v)^qdv \nonumber \\ 
&=\frac{1}{n}\left[-h_C(a_0\P_2, u_{j_0})\right]^{-p} \int_{\Omega_{C} \cap \Delta_{a_0\P_2, j_0}}\rho_C(a_0\P_2, v)^qdv \nonumber  \\ 
&>\frac{1}{n}\left[-h_C(\P_1, u_{j_0})\right]^{-p} \int_{\Omega_{C} \cap \Delta_{\P_1, j_0}}\rho_C(\P_1, v)^qdv = c_{j_0, 1},\label{p=q--1}
\end{align} where in the last inequality, we also use the fact that $\rho_C(\P_1, v)=\rho_C(a_0\P_2, v)$ for $v\in \Omega_{C} \cap \Delta_{\P_1, j_0}.$ That is, a contradiction  with $c_{j_0, 2}=c_{j_0, 1}$ is obtained, and thus  $\P_1$ and $\P_2$ are dilates of each other if $p=q$. 

Now for $p<q$, we need to prove $\P_1=\P_2$. Assume the opposite, i.e., $\P_1\neq \P_2$. It is clear that, due to \eqref{homCpq}, $\P_1$ and $\P_2$ are not dilates of each other. Again, let $a_0>0$ be such that $\P_1\subsetneq a_0\P_2$ and $\Sigma$ is a nonempty proper subset of $\{1, \cdots, m\}.$ Following 
the calculation in \eqref{p=q--1}, by $ c_{j_0, 1}= c_{j_0, 2}$, one gets, for $p<q$,
\begin{align*}
c_{j_0, 2} &= \frac{a_0^{p-q} }{n}\left[-h_C(a_0\P_2, u_{j_0})\right]^{-p} \int_{\Omega_{C} \cap \Delta_{a_0\P_2, j_0}}\rho_C(a_0\P_2, v)^qdv \nonumber  \\ &>\frac{a_0^{p-q} }{n}\left[-h_C(\P_1, u_{j_0})\right]^{-p} \int_{\Omega_{C} \cap \Delta_{\P_1, j_0}}\rho_C(\P_1, v)^qdv =a_0^{p-q}  c_{j_0, 1}.
\end{align*} This gives $a_0>1$, due to $p<q$. Thus, $\P_1\subsetneq a_0\P_2\subsetneq \P_2.$ 

By switching the roles of $\P_1$ and $\P_2$, one should get $\P_2\subsetneq \P_1$. This is impossible in view of  $\P_1\subsetneq \P_2$. Therefore, we prove  $\P_1=\P_2$. 
\end{proof}
 
Our last result is the following uniqueness result regarding the $L_p$ dual Minkowski problem of $(C, p, q)$-close sets which have enough smoothness in $\Omega_C$. Chen and Tu \cite{CT24}, independently, studied the uniqueness result for $q\geq p\geq 1$ by assuming that the support functions are $C^{\infty}$ in the smooth and strictly convex domain $\Oc$. 

\begin{theorem}\label{uni&C2} 
Let $p<q$ be two real numbers. If $\A_1, \A_2$ are two $(C, p, q)$-close sets such that both $\widetilde{C}_{p, q}(\A_1, \cdot)$ and $\widetilde{C}_{p, q}(\A_2, \cdot)$ are positive measures on $\Oc$, $\widetilde{C}_{p, q}(\A_1, \cdot)=\widetilde{C}_{p, q}(\A_2, \cdot)$ on $\Omega_{C^\circ}$, and their support functions are strictly negative and are in $C^2$, then $\A_1=\A_2$.
\end{theorem} 

In order to prove Theorem \ref{uni&C2}, we need some preparation of notations and background, which can be found in \cite{Sch14} . Denote by $\overline{\mathbb{R}}=\mathbb{R}\cup\{\infty\}$. We say $f: \mathbb{R}^n \rightarrow \overline{\mathbb{R}}$ a convex function, if for all $0\leq \lambda\leq 1$ and $x, y \in \mathbb{R}^n$, 
\begin{equation*}
f(\lambda x+(1-\lambda)y)\leq \lambda f(x)+(1-\lambda)f(y).
\end{equation*} The domain of $f$ is defined by $\mathrm{dom}(f) =\{x\in \mathbb{R}^n: f(x)<\infty\}$. The subdifferential of $f$ at $x \in  \mathrm{int} (\mathrm{dom}(f))$ is given by
\begin{equation*}
\partial f(x)=\{ z \in \mathbb{R}^n: f(y)\geq f(x)+z \cdot (y-x)\,\,\, \text{for all}\,\,\, y \in  \mathbb{R}^n\}.
\end{equation*} If $\partial f(x)$ is a singleton set, then $f$ is said to be differentiable at $x$. In this case, we use $\nabla f(x)$ to denote the gradient function of $f$ at $x$. Thus, 
\begin{equation*}
f(y)=f(x)+\nabla f(x) \cdot (y-x)+o(|y-x|)\quad \text{as} \quad y\rightarrow x. 
\end{equation*}

Let $\A\in \Ln$ be a $C$-compatible set. Define the support set of $\A$ at $u\in \Oc$ by $$F(\A, u)=\partial \A \cap H(\A, h_C(\A, u)).$$ As in \cite[Theorem 1.7.4]{Sch14} for convex bodies, one can also get, $\partial h_C(\A, u)=F(\A, u)$ for all $u \in \Omega_{C^\circ}$. Consequently,   $h_C(\A, u)$ is differentiable at $u \in \Omega_{C^\circ}$ if and only if  $F(\A, u)$ is a singleton set, say $F(\A, u)=\{z\}$.  Then, $z=\nabla h_C(\A, u)$.  

Let $\A$ be a  $(C, p, q)$-close set such that $h_C(\A, \cdot) $ is $ C^2$ and is strictly negative. Thus, $\partial \A\subset \text{int}C$ and, for $u \in \Omega_{C^\circ}$,  there is unique 
$x \in \partial \A$ such that 
\begin{equation}\label{gradient}
x=\nabla h_C(\A, u).
\end{equation} Recall that $H^{-}_t=H^{-}(\xi, t)$ with $\xi \in \Omega_{C}$, such that $x \cdot \xi>0$ for all $x \in C\setminus\{o\}$. Thus, there exists $t_0>0$ such that $\A \cap H^-_{t_0}\neq \emptyset$. Take an increasing sequence $\{t_l\}_{l \in \mathbb{N}}$ with $t_1\geq t_0$ satisfying $t_l \rightarrow \infty$ as $l \rightarrow \infty$. For each $l \in \mathbb{N}$,  $\A \cap H^{-}_{t_l} \subset C_{t_l}=C\cap H^{-}_{t_l}$ forms a convex body. Together with the argument in the proof of \cite[Theorem 6.5]{CVPDEII} (see the second paragraph in page 28), one sees that the surface area measure $S_{n-1}(\A, \cdot)$ is absolutely continuous with respect to $\mathscr{H}^{n-1}$ with a continuous density $D(\A, \cdot)$ on $\Omega_{C^\circ}$ (one may need to use the fact that, for each Borel set $\eta \subseteq \Omega_{C^\circ}$ with its closure $\bar{\eta}\subset\Oc$, one can find a $l\in \mathbb{N}$ big enough, such that $S_{n-1}(\A, \eta)=S_{n-1}(\A \cap  H^{-}_{t_l}, \eta)$). Note that $D(\A, u)$ equals  the product of the principal radii of curvature of $\A$  at $u\in \Oc$. We denote by $\mathbb{H}(\A, u)$ the Hessian matrix of $h_{C}(\A, \cdot)$ at $u\in \Oc$. From \cite[page 31, Note 3]{Sch14}, by letting $f=h_{C}(\A, \cdot), Af=\mathbb{H}(\A,\cdot), y=u$ and $x = u_0$ for $u, u_0 \in \Oc$, we have 
$$ h_C(\A, u)=h_C(\A,u_0)+\nabla h_C(\A, u_0)(u-u_0)+\frac{1}{2}(u-u_0)^T \mathbb{H}(\A,u_0)(u-u_0)+o(|u-u_0|^2).$$ 

We can adapt the proof of \cite[Theorem 6.5]{CVPDEII} to prove Theorem \ref{uni&C2}. 
  
\begin{proof}[Proof of Theorem \ref{uni&C2}.]  Let $p<q$ be two real numbers. Let $\A_1, \A_2$ be two $(C, p, q)$-close sets such that $\widetilde{C}_{p, q}(\A_1, \cdot)=\widetilde{C}_{p, q}(\A_2, \cdot)$ on $\Omega_{C^\circ}$, and their support functions are strictly negative and are in $C^2$. As discussed above,  for $i=1, 2$, $\partial \A_i\subset \Omega_C$, and $S_{n-1}(\A_i, \cdot)$ are absolutely continuous with respect to the spherical Lebesgue measure with continuous density functions $D(\A_i, \cdot)$. It follows from \eqref{borel&qL}, \eqref{gradient} and $\widetilde{C}_{p, q}(\A_1, \cdot)=\widetilde{C}_{p, q}(\A_2,\cdot)$ that, for $u \in \Omega_{C^\circ}$, 
\begin{equation}\label{*}
[-h_C(\A_1, u)]^{1-p}|\nabla h_C(\A_1, u)|^{q-n}D(\A_1, u)=[-h_C(\A_2, u)]^{1-p}|\nabla h_C(\A_2, u)|^{q-n}D(\A_2, u).
\end{equation}
 
We need to prove that $\A_1=\A_2$. Assume the opposite, i.e., $\A_1\neq \A_2$. Due to the homogeneity (see \eqref{homCpq}), $\A_1$ and $\A_2$ are not dilated of each other. Thus, one can find a constant $\lambda_0>0$ such that $\A_3=\lambda_0\A_1\subsetneq \A_2$ satisfying the set   
\begin{equation}\label{u_0}
\Sigma_1=\{u\in \Oc: h_C(\A_3, u)=h_C(\A_2, u)\,\,\,\text{and}\,\,\, \nabla  h_C(\A_3, u)=\nabla h_C(\A_2, u)\}
\end{equation} is a nonempty proper subset of $\Oc$, due to the smoothness of $\partial \A_1$ and $\partial \A_2.$ 

Let $u_0\in \Sigma_1.$ We claim that $D(\A_3, u_0) \leq D(\A_2, u_0)$. Let $$U=\{u_0+tv: t>0\ \text{and}\ v\in \Omega_{C^\circ}\}\subset \mathrm{int}C^{\circ}$$ be a neighbourhood of $u_0$.
Define $G(u)=h_C(\A_2, u)-h_C(\A_3, u)$ for
$u\in U$. Then $G(u)\geq 0$ for all $u\in U$, $G(u_0)=0$ and $\nabla G(u_0)=0$ by \eqref{u_0}. Consequently, $G(u)$ attains local minimum at $u_0\in U$, and the Hessian of $G(\cdot)$ is positive semi-definite at $u_0\in U$. Hence, $v^T (\mathbb{H}(\A_2,u_0)-\mathbb{H}(\A_3, u_0))v\, \geq 0$ for all $v\in \Omega_{C^\circ}$. We can then use \cite[Corollary 2.5.2]{Sch14} to get that $u_0$ is an eigenvalue of both $\mathbb{H}(\A_2,u_0)$ and $\mathbb{H}(\A_3, u_0)$, and $D(\A_i, u_0)$ is the product of other nonzero eigenvalues of $\mathbb{H}(\A_i,u_0)$, $i=2, 3.$ This yields that $
D(\A_3, u_0)\leq D(\A_2, u_0).$   Together with \eqref{homCpq}, 
\eqref{*} and \eqref{u_0}, one has  
\begin{align*}
[-h_C(\A_1, u_0)]^{1-p}|\nabla h_C(\A_1, u_0)|^{q-n} D(\A_1, u_0) 
=&[-h_C(\A_2, u_0)]^{1-p}|\nabla h_C(\A_2, u_0)|^{q-n} D(\A_2, u_0)\\
=&[-h_C(\A_3, u_0)]^{1-p}|\nabla h_C(\A_3, u_0)|^{q-n} D(\A_2, u_0)\\
\geq& [-h_C(\A_3, u_0)]^{1-p}|\nabla h_C(\A_3, u_0)|^{q-n} D(\A_3, u_0)\\
=& \lambda_0^{q-p}[-h_C(\A_1, u_0)]^{1-p}|\nabla h_C(\A_1, u_0)|^{q-n} D(\A_1, u_0).
\end{align*} As $\widetilde{C}_{p ,q}(\A_1, \cdot)$ is a positive measure on $\Oc$, one gets $D(\A_1, u_0)>0$ and  then  $\lambda_0^{q-p} \leq 1$. Thus, $\lambda_0\leq 1$ as $q-p>0$, which further yields  $\A_1\subseteq \lambda_0\A_1 =\A_3 $. Due to $\A_3\subsetneq \A_2$, one gets $\A_1\subsetneq \A_2.$ 

By switching the roles of $\A_1$ and $\A_2$, one can also get $\A_2\subsetneq \A_1,$ which is impossible. This concludes  $\A_1=\A_2$ as desired. 
\end{proof}

\vskip 2mm \noindent  {\bf Acknowledgement.}   The research of DY is supported by an NSERC grant, Canada.

\noindent Wen Ai, \ \ {\tt wena@mun.ca} \\ 
Department of Mathematics and Statistics, Memorial
University of Newfoundland, St. John's, Newfoundland A1C 5S7, Canada.  
\vskip 2mm

\noindent Yunlong Yang, \ \ {\tt ylyang@dlmu.edu.cn} \\ 
School of Science, Dalian Maritime University, Dalian, 116026, PR China.  
\vskip 2mm 

\noindent Deping Ye, \ \ {\tt deping.ye@mun.ca}\\
Department of Mathematics and Statistics, Memorial
University of Newfoundland, St. John's, Newfoundland A1C 5S7, Canada.  
\end{document}